\documentclass{alggeom}
\usepackage{mathtools, microtype, mathrsfs, xfrac, amssymb, enumerate, xspace, stmaryrd, amsthm}
\usepackage[alphabetic]{amsrefs}
\usepackage{ascii}
\usepackage[latin1]{inputenc}
\usepackage{tikz-cd}
\usepackage[all,cmtip]{xy}
\usepackage{hyperref}
\hypersetup{
	colorlinks=true,
	linkcolor=blue,
	filecolor=magenta,      
	urlcolor=cyan,
}
\usepackage{faktor}
\usepackage{cleveref}
\usepackage{comment}
\usepackage{graphicx}
\usepackage{aliascnt}
\usepackage{float}

\numberwithin{equation}{section}

\numberwithin{subsection}{section}

\allowdisplaybreaks[1]

\newtheorem{theorem}{Theorem}[section]
\newtheorem{proposition}[theorem]{Proposition}
\newtheorem{proposition-definition}[theorem]
{Proposition-Definition}
\newtheorem{corollary}[theorem]{Corollary}
\newtheorem{lemma}[theorem]{Lemma}

\theoremstyle{definition}
\newtheorem{definition}[theorem]{Definition}
\newtheorem{notation}[theorem]{Notation}

\newtheorem{remark}[theorem]{Remark}

\theoremstyle{remark}

\newcommand\cA{\mathcal{A}} \newcommand\cB{\mathcal{B}}
\newcommand\cC{\mathcal{C}} 
\newcommand\cE{\mathcal{E}} \newcommand\cF{\mathcal{F}}
 \newcommand\cH{\mathcal{H}}
 
 \newcommand\cL{\mathcal{L}}
\newcommand\cM{\mathcal{M}} 
\newcommand\cO{\mathcal{O}}

\newcommand\cU{\mathcal{U}} 
 \newcommand\cX{\mathcal{X}}
 \newcommand\cZ{\mathcal{Z}}

\renewcommand\AA{\mathbb{A}} 

\newcommand\GG{\mathbb{G}} \newcommand\HH{\mathbb{H}}

 \newcommand\PP{\mathbb{P}}
\newcommand\QQ{\mathbb{Q}} 
 
 \newcommand\VV{\mathbb{V}}
 
 \newcommand\ZZ{\mathbb{Z}}

\newcommand\rma{\mathrm{a}}

\newcommand\rmm{\mathrm{m}}

\newcommand\arr{\ifinner\to\else\longrightarrow\fi}

\newcommand\arrto{\ifinner\mapsto\else\longmapsto\fi}

\renewcommand\H{\operatorname{H}}

\newcommand\into{\hookrightarrow}

\newcommand\im[1]{\operatorname{im}(#1)}

\def\displaytimes_#1{\mathrel{\mathop{\times}\limits_{#1}}}

\def\displayotimes_#1{\mathrel{\mathop{\bigotimes}\limits_{#1}}}

\newcommand\aut{\operatorname{Aut}}

\newcommand\id{\mathrm{id}}

\newlength{\ignora}

\renewcommand{\setminus}{\smallsetminus}

\newcommand{\gm}{\GG_{\rmm}}

\newcommand{\GL}{\mathrm{GL}}

\newcommand{\PGL}{\mathrm{PGL}}
\newcommand{\ga}{\GG_{\rma}}

\DeclareFontFamily{U}{mathx}{\hyphenchar\font45}
\DeclareFontShape{U}{mathx}{m}{n}{
	<5> <6> <7> <8> <9> <10>
	<10.95> <12> <14.4> <17.28> <20.74> <24.88>
	mathx10
}{}
\DeclareSymbolFont{mathx}{U}{mathx}{m}{n}
\DeclareFontSubstitution{U}{mathx}{m}{n}
\DeclareMathAccent{\widecheck}{0}{mathx}{"71}
\DeclareMathAccent{\wideparen}{0}{mathx}{"75}

\renewcommand{\epsilon}{\varepsilon}

\newcommand{\Mbar}{\overline{\cM}}

\newcommand{\Mtilde}{\widetilde{\mathcal M}}

\newcommand{\Ctilde}{\widetilde{\mathcal C}}

\newcommand{\Htilde}{\widetilde{\cH}}

\newcommand{\ThTilde}{\widetilde{\Theta}}

\newcommand{\Detilde}{\widetilde{\Delta}}

\newcommand{\ch}[1][*]{\operatorname{CH}^{#1}}

\newcommand{\mt}{\widetilde{\mathcal M}}

\begin{document}

\title{The (almost) integral Chow ring of $\Mbar_3$}
\author{Michele Pernice}
\email{m.pernice(at)kth.se}
\address{KTH, Room 1642, Lindstedtsv\"agen 25,
114 28, Stockholm }



\begin{abstract}
	This paper is the fourth in a series of four papers aiming to describe the (almost integral) Chow ring of $\Mbar_3$, the moduli stack of stable curves of genus $3$. In this paper, we finally compute the Chow ring of $\Mbar_3$ with $\ZZ[1/6]$-coefficients.
\end{abstract}

\maketitle
\section*{Introduction}

The geometry of the moduli spaces of curves has always been the subject of intensive investigations, because of its manifold implications, for instance in the study of families of curves. One of the main aspects of this investigation is the intersection theory of these spaces, which has both enumerative and geometrical implication. In his groundbreaking paper \cite{Mum}, Mumford introduced the intersection theory with rational coefficients for the moduli spaces of stable curves. He also computed the Chow ring (with rational coefficients) of $\Mbar_2$, the moduli space of stable genus $2$ curves. While the rational Chow ring of $\cM_g$, the moduli space of smooth curves, is known for $2\leq g\leq 9$ (\cite{Mum}, \cite{Fab}, \cite{Iza}, \cite{PenVak}, \cite{CanLar}), the computations in the stable case are much harder. The complete description of the rational Chow ring has been obtained only for genus $2$ by Mumford and for genus $3$ by Faber in \cite{Fab}. In his PhD thesis, Faber also computed the rational Chow ring of $\Mbar_{2,1}$, the moduli space of $1$-pointed stable curves of genus $2$.

Edidin and Graham introduced in \cite{EdGra} the intersection theory of global quotient stacks with integer coefficients. It is a more refined invariant but as expected, the computations for the Chow ring with integral coefficients of the moduli stack of curves are much harder than the ones with rational coefficients. To date, the only complete description for the integral Chow ring of the moduli stack of stable curves is the case of $\Mbar_2$, obtained by Larson in \cite{Lar} and subsequently with a different strategy by Di Lorenzo and Vistoli in \cite{DiLorVis}. It is also worth mentioning the result of Di Lorenzo, Pernice and Vistoli regarding the integral Chow ring of $\Mbar_{2,1}$, see \cite{DiLorPerVis}.

The aim of this series of paper is to describe the Chow ring with $\ZZ[1/6]$-coefficients of the moduli stack $\Mbar_3$ of stable genus $3$ curves. This provides a refinement of the result of Faber with a completely indipendent method. The approach is a generalization of the one used in \cite{DiLorPerVis}: we introduce an Artin stack, which is called the stack of $A_r$-stable curves, where we allow curves with $A_r$-singularities to appear. The idea is to compute the Chow ring of this newly introduced stack in the genus $3$ case and then, using localization sequence, find a description for the Chow ring of $\Mbar_3$. The stack $\Mtilde_{g,n}$ introduced in \cite{DiLorPerVis} is cointained as an open substack inside our stack. We state the main theorem. 

\begin{theorem}
	The Chow ring of $\Mbar_3$ with $\ZZ[1/6]$-coefficients is the quotient of the graded polynomial algebra 
	$$\ZZ[1/6,\lambda_1,\lambda_2,\lambda_3,\delta_{1},\delta_{1,1},\delta_{1,1,1},H]$$
	where 
	\begin{itemize}
		\item[] $\lambda_1,\delta_1,H$ have degree $1$, \item[]$\lambda_2,\delta_{1,1}$ have degree $2$, \item[]$\lambda_3,\delta_{1,1,1}$ have degree $3$.
	\end{itemize}
	The quotient ideal is generated by 15 homogeneous relations, where
	\begin{itemize}
		\item  $1$ of them is in codimension $2$,
		\item  $5$ of them are in codimension $3$,
		\item  $8$ of them are in codimension $4$,
		\item  $1$ of them is in codimension $5$.
	\end{itemize}  
\end{theorem}

At the end of this paper, we explain how to compare the result of Faber with ours and comment about the information we lose in the process of tensoring with rational coefficients.

\subsection*{Stable $A_r$-curves and the strategy of the proof}

The strategy for the computation is the same used in \cite{DiLorPerVis} for the integral Chow ring of $\Mbar_{2,1}$. Suppose we have a closed immersion of smooth stacks $\cZ \into \cX$ and we know how to compute the Chow rings of $\cZ$ and $\cX \setminus \cZ$. We would like to use the well-known localization sequence
$$ 
 \ch(\cZ) \rightarrow \ch(\cX) \rightarrow \ch(\cX \setminus \cZ) \rightarrow 0
$$
to get the complete description of the Chow ring of $\cX$. To this end, we  make use of a patching technique which is at the heart of the Borel-Atiyah-Seigel-Quillen localization theorem, which has been used by many authors in the study of equivariant cohomology, equivariant Chow ring and equivariant K-theory. See the introduction of \cite{DiLorVis} for a more detailed discussion. 

However, without information regarding the kernel of the pushforward of the closed immersion $\cZ \into \cX$, there is no hope to get the complete description of the Chow ring of $\cX$. If the top Chern class of the normal bundle $N_{\cZ|\cX}$ is a non-zero divisor inside the Chow ring of $\cZ$, we can recover $\ch(\cX)$ from $\ch(\cX \setminus \cZ)$, $\ch(\cZ)$ and some patching information. We will refer to the condition on the top Chern class of the normal bundle as the \emph{gluing condition}. The gluing condition implies that the troublesome kernel is trivial. See \Cref{lem:gluing} for a more detailed statement. 

Unfortunately, there is no hope that this condition is verified if $\cZ$ is a Deligne-Mumford separated stack, because in this hypothesis the integral Chow ring is torsion above the dimension of the stack. This follows from Theorem 3.2 of \cite{EdGra}. This is exactly the reason that motivated the authors of  \cite{DiLorPerVis} to introduce the stack of cuspidal stable curves, which is not a Deligne-Mumford separated stack because it has some positive-dimensional affine stabilizers. However, introducing cusps is not enough in the case of $\Mbar_3$ to have the gluing condition verified (for the stratification we choose). This motivated us to introduce a generalization of the moduli stack of cuspidal stable curves, allowing curves with $A_r$-singular points to appear in our stack. They are a natural generalization of nodes and cusps, and \'etale locally are plane singularities described by the equation $y^2=x^{n+1}$.

\subsection*{Future Prospects}

As pointed out in the introduction of \cite{DiLorPerVis},  the limitations of this strategy are not clear. It seems that the more singularities we add, the more it is likely that the gluing condition is verified. However, adding more singularities implies that we  have to eventually compute the relations coming from such loci, which can be hard. Moreover, we are left with a difficult problem, namely to find the right stratification for these newly introduced stacks. We hope that this strategy will be useful to study the intersection theory of $\Mbar_{3,1}$ or $\Mbar_{4}$. Moreover, we believe that our approach can be used to obtain a complete description for the integral Chow ring of $\Mbar_3$. We have not verified the gluing condition with integer coefficients because we do not know the integral Chow ring of some of the strata. However, one can try to prove alternative descriptions for these strata, for instance using weighted blowups, and compute their integral Chow ring using these descriptions. See \cite{Ink} for an example.

\subsection*{Outline of the paper}
This is the fourth (and last) paper in the series. It focuses on computing the Chow ring of the moduli stack of stable curves of genus $3$.

Specifically, \Cref{sec:1} is dedicated to recall the theory of $A_r$-stable curves as discussed in \cite{Per1}, the theory of hyperelliptic $A_r$-stable curves as in \cite{Per2} and the setting introduced in \cite{Per3} for the computation of the Chow ring of $\Mtilde_3^7$. In particular, we state again the main results needed for this paper, such as the existence of the contraction morphism, the description of the hyperelliptic locus and the description of the Chow ring of $\Mtilde_3^7$. In the last part of the section, we discuss the strategy adopted for the computation of the Chow ring of $\Mbar_3$. Specifically, we introduce a stratification of the closed complement of $\Mbar_3$ inside $\Mtilde_3^7$ and explain why it is enough to reduce to study the locally closed strata of the stratification. 

In \Cref{sec:2}, we find a list of generators for the ideal of relation coming from the closed complement of $\Mbar_3$ inside $\Mtilde_3^7$. We prove infact that such ideal is generated only by the fundamental class of the locally closed strata introduced in \Cref{sec:1}, plus some relations coming from the open stratum. 

In \Cref{sec:3}, we explicitily compute the relations using the gluing lemma, restricting the fundamental classes of the locally closed strata to the stratification used in \cite{Per3} to compute the Chow ring of $\Mtilde_3^7$. This is the heavier part of the paper in terms of computations and it uses a lot of results and notations from \cite{Per3}.

Finally, we summarize the results in the main theorem in \Cref{sec:4}, where we discuss the connection between our description and Faber's one of the Chow ring of $\Mbar_3$. 

\section{Preliminaries and strategy}\label{sec:1}

In this section, we recall the definition of the moduli stack $\Mtilde_{g,n}^r$ parametrizing $n$-pointed $A_r$-stable curves of genus $g$ and some results regarding this moduli stack. For a more detailed treatment of the subject, see \cite{Per1}. Furthermore, we recall the definition of the moduli stack $\Htilde_g^r$ parametrizing hyperelliptic $A_r$-stable curves of genus $g$ and the alternative description using the theory of cyclic covers. For a more detailed treatment of the subject, see \cite{Per2}. Moreover, we recall the strategy used for the computation of the (almost) integral Chow ring of $\Mtilde_3^7$, in particular stating the gluing lemma (see \Cref{lem:gluing} and describing the stratification we used. For a more detailed treatment of the subject, see \cite{Per3}. Finally, we explain how to obtain the description of the (almost) integral Chow ring of $\Mbar_3$ from the one of $\Mtilde_3^7$.

\subsection*{Moduli stack of $A_r$-stable curves}
Fix a nonnegative integer $r$. Let $g$ be an integer with $g\geq 2$ and $n$ be a nonnegative integer.

\begin{definition}	
Let $k$ be an algebraically closed field and $C/k$ be a proper reduced connected one-dimensional scheme over $k$. We say the $C$ is a \emph{$A_r$-prestable curve} if it has at most $A_r$-singularity, i.e. for every $p\in C(k)$, we have an isomorphism
		$$ \widehat{\cO}_{C,p} \simeq k[[x,y]]/(y^2-x^{h+1}) $$ 
	with $ 0\leq h\leq r$. Furthermore, we say that $C$ is $A_r$-stable if it is $A_r$-prestable and the dualizing sheaf $\omega_C$ is ample. A $n$-pointed $A_r$-stable curve over $k$ is $A_r$-prestable curve together with $n$ smooth distinct closed points $p_1,\dots,p_n$ such that $\omega_C(p_1+\dots+p_n)$ is ample.
\end{definition}
\begin{remark}
	Notice that a $A_r$-prestable curve is l.c.i by definition, therefore the dualizing complex is in fact a line bundle. 
\end{remark}

We fix a base field $\kappa$ where all the primes smaller than $r+1$ are invertible. Every time we talk about genus, we intend arithmetic genus, unless specified otherwise. We recall a useful fact.

\begin{remark}\label{rem:genus-count}
	Let $C$ be a connected, reduced, one-dimensional, proper scheme over an algebraically closed field. Let $p$ be a rational point which is a singularity of $A_r$-type. We denote by $b:\widetilde{C}\arr C$ the partial normalization at the point $p$ and by $J_b$ the conductor ideal of $b$. Then a straightforward computation shows that 
	\begin{enumerate}
		\item if $r=2h$, then $g(C)=g(\widetilde{C})+h$;
		\item if $r=2h+1$ and $\widetilde{C}$ is connected, then $g(C)=g(\widetilde{C})+h+1$,
		\item if $r=2h+1$ and $\widetilde{C}$ is not connected, then $g(C)=g(\widetilde{C})+h$.
	\end{enumerate}
    If $\widetilde{C}$ is not connected, we say that $p$ is a separating point. Furthermore, Noether formula gives us that $b^*\omega_C \simeq \omega_{\widetilde{C}}(J_b^{\vee})$.
\end{remark}

 We can define $\Mtilde_{g,n}^r$ as the fibered category over $\kappa$-schemes  whose objects are the datum of $A_r$-stable curves over $S$ with $n$ distinct sections $p_1,\dots,p_n$ such that every geometric fiber over $S$ is a  $n$-pointed $A_r$-stable curve. These families are called \emph{$n$-pointed $A_r$-stable curves} over $S$. Morphisms are just morphisms of $n$-pointed curves.

We recall the following description of $\Mtilde_{g,n}^r$. See Theorem 2.2 of \cite{Per1} for the proof of the result. 

\begin{theorem}\label{theo:descr-quot}
	$\Mtilde_{g,n}^r$ is a smooth connected algebraic stack of finite type over $\kappa$. Furthermore, it is a quotient stack: that is, there exists a smooth quasi-projective scheme X with an action of $\GL_N$ for some positive $N$, such that 
	$ \Mtilde_{g,n}^r \simeq [X/\GL_N]$.
\end{theorem}

\begin{remark}\label{rem: max-sing}
Recall that we have an open embedding $\Mtilde_{g,n}^r  \subset \Mtilde_{g,n}^s$ for every $r\leq s$. Notice that $\Mtilde_{g,n}^r=\Mtilde_{g,n}^{2g+1}$ for every $r\geq 2g+1$. 
\end{remark}

The usual definition of the Hodge bundle extends to our setting. See Proposition 2.4 of \cite{Per1} for the proof. As a consequence we obtain a locally free sheaf $\HH_{g}$ of rank~$g$ on $\mt_{g, n}^r$, which is called \emph{Hodge bundle}. 

Furthermore, we recall the existence of the (minimal) contraction morphism. This is Theorem 2.5 of \cite{Per1}.

\begin{theorem}\label{theo:contrac}
	We have a morphism of algebraic stacks 
	$$ \gamma:\Mtilde_{g,n+1}^r \longrightarrow \Ctilde_{g,n}^r$$
	where $\Ctilde_{g,n}^r$ is the universal curve of $\Mtilde_{g,n}^r$. Furthermore, it is an open immersion and its image is the open locus $\Ctilde_{g,n}^{r,\leq 2}$ in $\Ctilde_{g,n}^r$ parametrizing $n$-pointed $A_r$-stable curves $(C,p_1,\dots,p_n)$ of genus $g$ and a (non-necessarily smooth) section $q$ such that $q$ is an $A_h$-singularity for $h\leq 2$.
\end{theorem}

\subsection*{The hyperelliptic locus}

We recall the definition of hyperelliptic $A_r$-stable curves.

\begin{definition}
	Let $C$ be an $A_r$-stable curve of genus $g$ over an algebraically closed field. We say that $C$ is hyperelliptic if there exists an involution $\sigma$ of $C$ such that the fixed locus of $\sigma$ is finite and the geometric categorical quotient, which is denoted by $Z$, is a reduced connected nodal curve of genus $0$. We  call the pair $(C,\sigma)$ a \emph{hyperelliptic $A_r$-stable curve} and such $\sigma$  is called a \emph{hyperelliptic involution}.
\end{definition}

We define $\Htilde_g^r$ as the following fibered category: its objects are the data of a pair $(C/S,\sigma)$ where $C/S$ is an $A_r$-stable curve over $S$ and $\sigma$ is an involution of $C$ over $S$ such that $(C_s,\sigma_s)$ is a $A_r$-stable hyperelliptic curve of genus $g$ for every geometric point $s \in S$. These are called \emph{hyperelliptic $A_r$-stable curves over $S$}. A morphism is a morphism of $A_r$-stable curves that commutes with the involutions. 

Now we introduce another description of $\Htilde_g^r$, useful for understanding the link with the smooth case, using cyclic covers of twisted curves. We refer to \cite{AbOlVis} for the theory of twisted nodal curves, although we consider only twisted curves with $\mu_2$ as stabilizers and with no markings.

\begin{definition}\label{def:hyp-A_r}
	Let $\cZ$ be a twisted nodal curve of genus $0$ over an algebraically closed field. We denote by $n_{\Gamma}$ the number of stacky points of $\Gamma$ and by $m_{\Gamma}$ the number of intersections of $\Gamma$ with the rest of the curve for every $\Gamma$ irreducible component of $\cZ$. Let $\cL$ be a line bundle on $\cZ$ and $i:\cL^{\otimes 2} \rightarrow \cO_{\cZ}$ be a morphism of $\cO_{\cZ}$-modules.  We denote by $g_{\Gamma}$ the quantity $n_{\Gamma}/2-1-\deg\cL\vert_{\Gamma}$. 
	\begin{itemize}
	\item[(a)] We say that $(\cL,i)$ is hyperelliptic if the following are true:
			\begin{itemize}
				\item[(a1)] the morphism $\cZ \rightarrow B\GG_m$ induced by $\cL$ is representable,
				\item[(a2)] $i^{\vee}$ does not vanish restricted to any stacky point.
			\end{itemize}
	\item[(b)] We say that $(\cL,i)$ is $A_r$-prestable and hyperelliptic of genus $g$ if $(\cL,i)$ is hyperelliptic, $\chi(\cL)=-g$ and the following are true:
		\begin{itemize}
			\item[(b1)] $i^{\vee}$ does not vanish restricted to any irreducible component of $\cZ$ or equivalently the morphism $i:\cL^{\otimes 2 }\rightarrow \cO_{\cZ}$ is injective,
			\item[(b2)] if $p$ is a non-stacky node and $i^{\vee}$ vanishes at $p$, then $r\geq 3$ and the vanishing locus $\VV(i^{\vee})_p$ of $i^{\vee}$ localized at $p$ is a Cartier divisor of length $2$;
			\item[(b3)] if $p$ is a smooth point and $i^{\vee}$ vanishes at $p$, then the vanishing locus $\VV(i^{\vee})_p$ of $i^{\vee}$ localized at $p$ has length at most $r+1$.
		\end{itemize}	
	\item[(c)] We say that $(\cL,i)$ is $A_r$-stable and hyperelliptic of genus $g$ if it is $A_r$-prestable and hyperelliptic of genus $g$ and the following are true for every irreducible component $\Gamma$ in $\cZ$:
			\begin{itemize}
				\item[(c1)] if $g_{\Gamma}=0$ then we have $2m_{\Gamma}-n_{\Gamma}\geq 3$,
				\item[(c2)] if $g_{\Gamma}=-1$ then we have
				$m_{\Gamma}\geq 3$ ($n_{\Gamma}=0$).
			\end{itemize}
\end{itemize}
\end{definition}

Let us define now the stack classifying these data. We denote by $\cC(2,g,r)$ the fibered category defined in the following way: the objects are triplet $(\cZ\rightarrow S,\cL,i)$ where $\cZ \rightarrow S$ is a family of twisted curves of genus $0$, $\cL$ is a line bundle of $\cZ$ and $i:\cL^{\otimes 2}\rightarrow \cO_{\cZ}$ is a morphism of $\cO_{\cZ}$-modules such that the restrictions $(\cL_s,i_s)$ to the geometric fibers over $S$ are $A_r$-stable and hyperelliptic of genus $g$. Morphisms are defined as in \cite{ArVis}. 

We recall the following result, which gives us an alternative description of $\Htilde_g^r$. See Proposition 2.14 and Proposition 2.21 of \cite{Per2} for the proof of the result.
\begin{proposition}\label{prop:descr-hyper}
	The fibered category $\cC(2,g,r)$ is isomorphic to $\Htilde_g^r$.
\end{proposition}

Finally, we recall the following theorem which is a conseguence of Proposition 2.23 and Section 3 of \cite{Per2}.

\begin{proposition}\label{prop:smooth-hyp}
	The moduli stack $\Htilde_g^r$ of $A_r$-stable hyperelliptic curves of genus $g$ is smooth and the open $\cH_g$ parametrizing smooth hyperelliptic curves is dense in $\Htilde_g^r$. In particular $\Htilde_g^r$ is connected. Moreover, the natural morphism 
	$$ \Htilde_g^r \longrightarrow \Mtilde_g^r$$
	defined by the association $(C,\sigma) \mapsto C$ is a closed embedding between smooth algebraic stacks.
\end{proposition}

\subsection*{The Chow ring of $\Mtilde_3^7$}

In this subsection, we describe briefly the stratey used for the computation of the Chow ring of $\Mtilde_3^7$. Every Chow ring is considered with $\ZZ[1/6]$-coefficients unless otherwise stated. Recall that our base field $\kappa$ has characteristic different from $2,3,5,7$.

First of all, we recall the gluing lemma. Let $i:\cZ\hookrightarrow\cX$ be a closed immersion of smooth global quotient stacks over $\kappa$ of codimension $d$ and let $\cU:=\cX\setminus \cZ$ be the open complement and $j:\cU \hookrightarrow \cX$ be the open immersion. It is straightforward to see that the pullback morphism $i^*:\ch(\cX)\rightarrow \ch(\cZ)$ induces a morphism $ \ch(\cU) \rightarrow \ch(\cZ)/(c_d(N_{\cZ|\cX}))$, where $N_{\cZ|\cX}$ is the normal bundle of the closed immersion. This morphism is denoted by $i^*$ by abuse of notation. 

Therefore, we have the following commutative diagram of rings:
$$
\begin{tikzcd}
\ch(\cX) \arrow[d, "j^*"] \arrow[rr, "i^*"] &  & \ch(\cZ) \arrow[d, "q"]     \\
\ch(\cU) \arrow[rr, "i^*"]                  &  & \frac{\ch(\cZ)}{(c_d(N_{\cZ|\cX}))}
\end{tikzcd}
$$
where $q$ is just the quotient morphism.

\begin{lemma}\label{lem:gluing}
  In the situation above, the induced map 
  $$\zeta: \ch(\cX)\longrightarrow \ch(\cZ)\times_\frac{\ch(\cZ)}{(c_d(N_{\cZ|\cX}))} \ch(\cU)$$
  is surjective and $\ker \zeta= i_* {\rm Ann}(c_d(N_{\cZ|\cX}))$. In particular, if $c_d(N_{\cZ|\cX})$ is a non-zero divisor in $\ch(\cZ)$, then $\zeta$ is an isomorphism. 
 \end{lemma}

From now on, we refer to the condition \emph{$c_d(N_{\cZ|\cX})$ is not a zero divisor} as the gluing condition.

We can apply \Cref{lem:gluing} to the following stratification of $\Mtilde_3^7$:
$$
\begin{tikzcd}
	&                            & \Htilde_3^7 \arrow[rd] &             \\
	{\Detilde_{1,1,1}} \arrow[r] & {\Detilde_{1,1}} \arrow[r] & \Detilde_1 \arrow[r]   & \Mtilde_3^7,
\end{tikzcd}
$$

where $\Detilde_1$ (respectively $\Detilde_{1,1}$, $\Detilde_{1,1,1}$) is the moduli stack parametrizing $A_7$-stable curves of genus $3$ with a least one (respectively two, three) separating nodes. The diagram above represents the poset associated to the stratification.  More precisely,our approach focuses firstly on the computation of the Chow ring of $\Mtilde_3^7 \setminus \Detilde_1$. We compute the Chow ring of $\Htilde_3^7 \setminus \Detilde_1$, then we apply the gluing lemma to $\Mtilde_3^7 \setminus (\Htilde_3^7 \cup \Detilde_1)$ and $\Htilde_3^7 \setminus \Detilde_1$ to get a description for the Chow ring of $\Mtilde_3^7 \setminus \Detilde_1$. Furthermore, we apply \Cref{lem:gluing} to $\Detilde_1 \setminus \Detilde_{1,1}$ and $\Mtilde_3^7\setminus \Detilde_1$ to describe the Chow ring of $\Mtilde_3^7 \setminus \Detilde_{1,1}$, and then apply it again to $\Mtilde_3^7 \setminus \Detilde_{1,1}$ and $\Detilde_{1,1} \setminus\Detilde_{1,1,1}$. The same procedure allows us to glue also $\Detilde_{1,1,1}$ and get the description of the Chow ring of $\Mtilde_3^7$.
 
For a more precise description of the strata and for the computations needed, see \cite{Per3}. We recall the main theorem of \cite{Per3}, which describes $\ch(\Mtilde_3^7)$ as a $\ZZ[1/6]$-algebra of finite type. 

\begin{theorem}\label{theo:main}
	We have the following isomorphism 
	$$ \ch(\Mtilde_3)\simeq \ZZ[1/6][\lambda_1,\lambda_2,\lambda_3,H,\delta_1,\delta_{1,1},\delta_{1,1,1}]/I$$
	where $I$ is generated by the following relations:
	\begin{itemize}
		\item $k_h$, which comes from the generator of $\ker i_H^*$, where $i_H: \Htilde_3\setminus \Detilde_1 \into \Mtilde_3\setminus \Detilde_1$;
		\item $k_{1}(1)$ and $k_1(2)$, which come from the two generators of $\ker i_{1}^*$ where $i_{1}: \Detilde_1\setminus \Detilde_{1,1} \into \Mtilde_3\setminus \Detilde_{1,1}$;
		\item $k_{1,1}(1)$, $k_{1,1}(2)$ and $k_{1,1}(3)$, which come from the three generators of $\ker i_{1,1}^*$ where $i_{1,1}: \Detilde_{1,1}\setminus \Detilde_{1,1,1} \into \Mtilde_3\setminus \Detilde_{1,1,1}$;
		\item $k_{1,1,1}(1)$, $k_{1,1,1}(2)$, $k_{1,1,1}(3)$ and $k_{1,1,1}(4)$, which come from the four generators of $\ker i_{1,1,1}^*$ where $i_{1,1,1}: \Detilde_{1,1,1} \into \Mtilde_3$;
		\item $m(1)$, $m(2)$, $m(3)$ and $r$, which are the litings of the generators of the relations of the open stratum $\Mtilde_3\setminus (\Htilde_3 \cup \Detilde_1)$;
		\item $h(1)$, $h(2)$ and $h(3)$, which are the liftings of the generators of the relations of the stratum $\Htilde_3 \setminus \Detilde_1$; 
		\item $d_1(1)$, which is the lifting of the generator of the relations of the stratum $\Detilde_1\setminus \Detilde_{1,1}$.
	\end{itemize}
	Furthemore, $h(2)$, $h(3)$ and $d_1(1)$ are in the ideal generated by the other relations. 
\end{theorem}

We have the following geometric description of the generators: 
\begin{itemize}
 \item[$\bullet$] $\lambda_i$ is the $i$-th Chern class of the Hodge bundle $\widetilde{\HH}$ of $\Mtilde_3^7$ for $i=1,2,3$;
 \item[$\bullet$] $H$ is the fundamental class of the hyperelliptic locus $\Htilde_3^7$ in $\Mtilde_3^7$;
 \item[$\bullet$] $\delta_1$ (respectively $\delta_{1,1}$, $\delta_{1,1,1}$) is the fundamental class of the closed substack $\Detilde_1$ (respectively $\Detilde_{1,1}$, $\Detilde_{1,1,1}$).
\end{itemize}

For an explicit description of the relations, see Remark 7.4 of \cite{Per3}.

\subsection{Strategy of the computation}

The final goal of this paper is to compute the Chow ring of $\Mbar_3$.  Since we have that $\Mbar_3$ is an open inside the $\Mtilde_3^7$, we can use the localization exact sequence
$$ \ch(\Mtilde_3^7\setminus \Mbar_3) \rightarrow \ch(\Mtilde_3^7) \rightarrow \ch(\Mbar_3) \rightarrow 0$$ 
to compute the Chow ring of $\Mbar_3$ as the quotient of the Chow ring of $\Mtilde_3^7$ by the ideal generated by the image of the pushforward of the closed immersion $\Mtilde_3^7\setminus \Mbar_3 \hookrightarrow \Mtilde_3^7$. The paper is dedicated to computing this specific ideal.

We recall the notation used in Section 4 of \cite{Per1}. Let $g\geq 2$ and $r\geq 1$ be two integers and $\kappa$ be the base field of characteristic either 0 or  greater than $2g+1$. We recall the sequence of open subset (see \Cref{rem: max-sing})
$$ \Mtilde_g^0 \subset \Mtilde_g^1 \subset \dots \subset \Mtilde_g^r$$ 
and we define $\widetilde{\cA}_{\geq n}:=\Mtilde_g^r\setminus \Mtilde_g^{n-1}$ for $n=0,\dots,r+1$ setting $\Mtilde_g^{-1}:=\emptyset$. By construction, $\Mbar_g$ is equal to the open $\Mtilde_g^1=\Mtilde_g^r \setminus \widetilde{\cA}_{\geq 2}$. In our setting, we need to describe the ideal generated by the image of the pushforward of the closed embedding $\widetilde{\cA}_{\geq 2} \hookrightarrow \Mtilde_3^7$.

We now introduce an alternative to $\widetilde{\cA}_{\geq n}$ which is easier to describe. Suppose $n$ is a positive integer smaller or equal than $r$ and let $\cA_{\geq n}$ be the substack of the universal curve $\Ctilde_g^r$ of $\Mtilde_g^r$ parametrizing pairs $(C/S,p)$ where $p$ is a section whose geometric fibers over $S$ are $A_r$-singularities for $r\geq n$. We give to $\cA_{\geq n}$ the structure of closed substack of $\Ctilde_g^r$ inductively on $n$. Clearly if $n=0$ we have $\cA_{\geq 0}=\Ctilde_g^r$. To define $\cA_{\geq 1}$, we need to find the stack-theoretic structure of the singular locus of the natural morphism $\Ctilde_g^r \rightarrow \Mtilde_g^r$. This is standard and it can be done by taking the zero locus of the $1$-st Fitting ideal of $\Omega_{\Ctilde_g^r|\Mtilde_g^r}$. We have that $\cA_{\geq 1}\rightarrow  \Mtilde_g^r$ is now finite and it is unramified over the nodes, while it ramifies over the more complicated singularities. Therefore, we can denote by $\cA_{\geq 2}$ the substack of $\cA_{\geq 1}$ defined by the $0$-th Fitting ideal of $\Omega_{\cA_{\geq 1}|\Mtilde_g^r}$. A local computation shows us that $\cA_{\geq 2} \rightarrow \Mtilde_g^r$ is unramified over the locus of $A_2$-singularities and ramified elsewhere. Inductively, we can iterate this procedure considering the $0$-th Fitting ideal of $\Omega_{\cA_{\geq n-1}|\Mtilde_g^r}$ to define $\cA_{\geq n}$. 

A local computation shows that the geometric points of $\cA_{\geq n}$ are exactly the pairs $(C,p)$ such that $p$ is an $A_{n'}$-singularity for $n\leq n'\leq r$.

Let us define $\cA_n:=\cA_{\geq n}\setminus \cA_{\geq n+1}$ for  $n=0,\dots,r-1$. We have a stratification of $\cA_{\geq 2}$
$$ \cA_{r}=\cA_{\geq r} \subset \cA_{\geq r-1} \subset \dots \subset \cA_{\geq 2}$$ 
where the $\cA_n$'s are the associated locally closed strata for $n=2,\dots, r$. 

The first reason we choose to work with $\cA_{\geq n}$ instead of $\widetilde{\cA}_{\geq n}$ is the smoothness of the locally closed substack $\cA_n$ of $\Ctilde_g^r$.

\begin{proposition}
	The stack $\cA_n$ is smooth.
\end{proposition}

See Proposition 4.1 of \cite{Per1} for the proof.

The second reason is that we have an explicit description of $\cA_n$ that allows us to compute their Chow rings. In particular we have the following results when $n$ is an even positive integer.

\begin{proposition}\label{prop:descr-an-pari}
	$\cA_{2m}$ is an affine bundle of dimension $m-1$ over the stack $\Mtilde^r_{g-m,[2m]} $ for $m\geq 1$.
\end{proposition}

For the proof of the previous result, see Corollary 4.12 of \cite{Per1}. Moreover, the stack $\Mtilde_{g-m,[2m]}^r$ is described in Proposition 4.9 of \cite{Per1}. 

When $n$ is odd, things are a bit more complicated. In fact, $\cA_{2m-1}$ is the disjoint union of several connected components, namely $\cA_{2m-1}^{\rm ns}$ and $\cA_{2m-1}^i$ for $0\leq i\leq (g-m+1)/2$. The stack $\cA_{2m-1}^{\rm ns}$ parametrizes pairs $(C,p)$ such that the partial normalization of $C$ at $p$ is connected, whereas $\cA_{2m-1}^{\rm ns}$ parametrizes pairs $(C,p)$ such that the partial normalization of $C$ at $p$ is the disjoint union of two components of genus $i$ and $g-m-i+1$. For a more detailed treatment of the subject, see Subsection 4.1 in \cite{Per1}. 

We have the following descriptions in the odd case. 

\begin{proposition}\label{prop:descr-an-disp-ns}
 We have a finite \'etale cover of degree $2$
 $$ F^{\rm ns}: I_{2m-1}^{\rm ns} \longrightarrow \cA_{2m-1}^{\rm ns}$$ 
 where $I_{2m-1}^{\rm ns}$ is a $\gm \ltimes U$-torsor over $\Mtilde_{g-m,2[m]}$ and $U$ is an algebraic group whose underlying scheme is isomorphic to $\AA^{m-2}$.  
\end{proposition}

For the proof of the previous result, see Proposition 4.17 of \cite{Per1}. Moreover, the stack $\Mtilde_{g-m,2[m]}$ is described in Proposition 4.15 of \cite{Per1}.

A similar statement is true for the other connected components. For the proof of the following result, see Proposition 4.18 of \cite{Per1}.

\begin{proposition}\label{prop:descr-an-disp-i}
 We have morphism 
 $$ F^i:I_{2m-1}^{i} \longrightarrow \cA_{2m-1}^{ i}$$ 
 where $I_{2m-1}^{i}$ is a $\gm \ltimes U$-torsor over $\Mtilde_{i,[m]} \times \Mtilde_{g-i-m+1,[m]}$ and $U$ is an algebraic group whose underlying scheme is isomorphic to $\AA^{m-2}$. If $2i=g-m+1$, then $F^i$ is finite \'etale of degree $2$; otherwise it is an isomorphism. 
\end{proposition}

Finally, we explain why we can reduce to study the stacks $\cA_n$ and their Chow rings. We have the following result.

\begin{proposition}
	In the case of $g=3$, the functor forgetting the section gives us a natural morphism
	$$ \cA_{\geq n} \rightarrow \widetilde{\cA}_{\geq n}$$
	which is finite birational and it is surjective at the level of Chow groups (after inverting $6$) for every $n\leq 7$.
\end{proposition} 

\begin{proof}
	It follows from the fact that every $A_r$-stable curve of genus $3$ has at most three singularity of type $A_n$ for $n\geq 2$. 
\end{proof}

Consider now the proper morphisms 
$$ \rho_{\geq n}:\cA_{\geq n} \longrightarrow \Mtilde_g^r$$ 
and their restrictions to $\cA_n$
$$ \rho_n :\cA_n \longrightarrow \Mtilde_g^r\setminus \widetilde{\cA}_{\geq n+1}$$
which is still proper; let $\{f_i\}_{i \in I_n}$ be a set of elements of $\ch(\cA_n)$ indexed by some set $I_n$ such that $\im{\rho_{n,*}}$ is generated by the set $\{\rho_{n,*}(f_i)\}_{i \in I_n}$. We choose a lifting $\tilde{f}_i$ of every $f_i$ to the Chow group of $\cA_{\geq n}$ for every $n=2,\dots,r$ and every $i \in I_n$. We have the following result.

\begin{lemma}\label{lem:strata}
	In the setting above, we have that $\im{\rho_{\geq 2,*}}$ is generated by $\{\rho_{\geq n,*}(\tilde{f}_i)\}_{\forall n, \forall i \in I_n}$
\end{lemma}

\begin{proof}
	This is a direct conseguence of Lemma 3.3 of \cite{DiLorFulVis}.
\end{proof}

The previous lemma implies that we need to focus on finding the generators of the relations coming from the strata $\cA_n$ for $n=2,\dots,7$. Therefore in the next section we study the morphism $\rho_n$ and we prove that $\rho_n^*$ is surjective at the level of Chow rings for every $n\geq 3$. The same is not true for $n=2$ but we describe geometrically the generators of the image of $\rho_{2,*}$. 

\section{The relations from the $A_n$-strata}\label{sec:2}

In this section, we are going to describe abstractly the generators of $\rho_{\geq2,*}$. \Cref{lem:strata} assures us it is enough to find the generators of $\rho_{n,*}$ for all $n=2,\dots,7$ and then lift them to $\Mtilde_3^7$. Recall that the characteristic of $\kappa$ is either 0 or greater than $7$. For the rest of the paper, if we do not specify the value of $r$ for a moduli stack of $A_r$-stable curves, it means that we are considering the biggest moduli stack we can form considering $A_r$-singularities. Because of \Cref{rem:genus-count}, we have that $r$ is bounded from above by a function of the genus. Following this notation, $\Mtilde_3$ is the moduli stack $\Mtilde_3^7$ and the same is true for the hyperelliptic locus $\Htilde_3$ and the universal curve $\Ctilde_3$. However, notice that while $\Mtilde_{1,1}$ is equal to $\Mtilde_{1,1}^2$, the notation $\Mtilde_{1,2}$ does not stand for $\Mtilde_{1,2}^2$ but for $\Mtilde_{1,2}^3$ as $\Mtilde_{1,2}^3 \setminus \Mtilde_{1,2}^2 \neq \emptyset$.

Before we start with the abstract computations, we need to recall some notations and results from Section 3 of \cite{Per1}. 

\begin{notation}\label{not:algebras}
 We denote by $\cF^c_m$ the moduli stack parametrizing pointed curvilinear algebras, i.e. pairs $(A,p)$ over an algebraically closed field $k$ where $A$ is a local Artinian algebra over $k$ of lenght $m$, $p$ is the section associated to the maximal ideal and $\dim_{k}m_p/m_p^2 \leq 1$. Furthermore, we denote by $\cE_{m,d}^c$ the moduli stack parametrizing finite flat extensions $(A,p) \hookrightarrow (B,q)$ of degree $d$ of pointed curvilinear algebras where $(A,p)$ is in $\cF^c_m$ (and therefore $(B,q)$ is in $\cF_{dm}^c$). 

Corollary 3.5 of \cite{Per1} gives us an isomorphism between $\cF^c_m$ and the classifying stack of $G_m$, where $G_m$ is the automorphism group of the standard curvilinear algebra $k[t]/(t^m)$ with the standard section $q_0$ defined by $t\mapsto 0$. Notice that every automorphism of the pair $(k[t]/(t^m),q_0)$ is determined by its value on $t$, therefore it follows easily that $G_m$ is isomorphic to $\gm \ltimes U$ where $U$ is a unipotent group (whose underlying scheme is an affine space of dimension $m-2$).

Proposition 3.7 of \cite{Per1} gives us a description of $\cE_{m,d}$ as a quotient stack of a scheme $E_{m,d}$ by the action of $G_m\times G_{md}$. The scheme $E_{m,d}$ can be described as the set of all finite flat extensions $\phi_d: k[t]/(t^m)\hookrightarrow k[t]/(t^{dm})$ of degree $d$ and the action of $G_m \times G_{dm}$ is the natural one. Again, $\phi_d$ is determined by its value on $t$ and it is easy to see that $\phi_d(t)$ is a polynomial of the form $$ a t^d + b_1 t^{d+1} + \dots + b_{d(m-1)-1} t^{dm-1}$$ 
where $a\neq 0$. It follows that $E_{m,d}$ is isomorphic to $(\AA^1 \setminus 0)\times \AA^{d(m-1)-1}$.

Finally, we recall that the morphism
$$ \cE_{m,d}^c \longrightarrow \cF^c_{dm}$$ 
is an affine bundle because $[E_{m,d}/G_m]$ is isomorphic to an affine space $V_{m,d}$ of dimension $(m-1)(d-1)$ described as 
$$V_{m,d}:=\{ f \in E_{m,d}|\, a_0=1, \, a_{kd}(f)=0\quad {\rm for }\quad k=1,\dots,m-2\}$$
where $a_{l}(f)$ is the coefficient of $t^{l}$ of the polynomial $f(t)$. See Lemma 3.8 of \cite{Per1}. 

\end{notation}

\subsection*{Generators for the image of $\rho_{n,*}$ if $n$ is even}

 We are going to prove that the morphism 
$$ \rho_n^*: \ch(\Mtilde_3^7) \longrightarrow \ch(\cA_n)$$ 
is surjective for $n=4,6$.

\begin{proposition}\label{prop:rho-6-surj}
	The morphism $\rho_6^*$ is surjective.
\end{proposition}

\begin{proof}
	We start by considering the isomorphism proved in Corollary 4.12 of \cite{Per1} (see \Cref{prop:descr-an-pari}) in case $m=3$ and $g=3$. It follows from Proposition 4.9 of \cite{Per1} that  
	$$\cA_6 \simeq [V_{3,3}/\gm \ltimes \ga];$$
    see \Cref{not:algebras} for a description of $V_{3,3}$. As a matter of fact, we proved that the following commutative diagram of stacks
	$$
	\begin{tikzcd}
		\cA_6 \arrow[d] \arrow[r]                         & {\cE_{3,2}^c\simeq [V/G_{6}]} \arrow[d] \\
		{\Mtilde_{0,[3]}\simeq \cB(\gm \ltimes \ga)} \arrow[r, "\cB f"] & \cF_{6}^c:=\cB G_6                     
	\end{tikzcd}
	$$
	is cartesian. The morphism $\cB f$ can be described as the morphism  of classifying stacks induced by the morphism of groups schemes
	$$ f:\gm \ltimes \ga\simeq \aut(\PP^1,\infty) \longrightarrow G_6$$
	defined by the association $\phi \mapsto \phi\otimes_{\cO_{\PP^1}} \cO_{\PP^1}/m_{\infty}^6$, where $m_{\infty}$ is the maximal ideal of the point $\infty$. For the definition of $G_6$, see \Cref{not:algebras}. A simple computation, using the explicit formula of the action of $G_6$ on $V$, shows that 
	$$ \cA_6\simeq [V/\gm \ltimes \ga] \simeq [\AA^1/\gm]$$
	where the action of $\gm$ on $\AA^1$ has weigth $-3$. More explicitly, if we identify $\cO_{\PP^1}/m_{\infty}^6$ with the algebra $\kappa[t]/(t^6)$, an element $\lambda \in \AA^1$ is equivalent to the inclusion of algebras $$\kappa[t]/(t^3) \into \kappa[t]/(t^6)$$ defined by the association $t \mapsto t^2+\lambda t^5$.

	We want to understand the pullback of the hyperelliptic locus, i.e. $\rho_6^*(H)$. It is clear that the locus $\rho_6^{-1}(\Htilde_3)$ is the locus in $[\AA^1/\gm]$ such that the involution $t \mapsto -t$ fixes the inclusion $t \mapsto t^2+\lambda t^5$. This implies $\lambda=0$ and therefore $\rho_6^*(\lambda)=-3s$ where $s$ is the generator of $\ch([\AA^1/\gm])$. This implies the surjectivity.
\end{proof}

\begin{proposition}
	The morphism $\rho_4^*$ is surjective.
\end{proposition} 

\begin{proof}
	Again, \Cref{prop:descr-an-pari} shows that $\cA_{4}$ is an affine bundle of $\Mtilde_{1,1}\times [\AA^1/\gm]$ and therefore 
	$$ \ch(\cA_{4})\simeq \ZZ[1/6,t,s]$$ 
	where $s$ is the generator of the Picard group of $[\AA^1/\gm]$ and $t$ is the $\psi$-class of $\Mtilde_{1,1}$ (which is a generator of the Chow ring of $\Mtilde_{1,1}$). Exactly as it happens for the pinching morphism described in \cite{DiLorPerVis}, we have $\rho_4^*(\delta_1)=s$ (see Lemma 5.9 of \cite{DiLorPerVis}). We need to compute now $\rho_4^*(H)$, where $H$ is the fundamental class of the hyperelliptic locus inside $\Mtilde_3^7$. Consider now the open immersion 
	$$ \cA_4\vert_{\Mtilde_{1,1}} \into \cA_4$$
	induced by the open immersion $\Mtilde_{1,1} \into \Mtilde_{1,1} \times [\AA^1/\gm]$. We have that 
	$$ \ch(\cA_4\vert_{\Mtilde_{1,1}})=\ZZ[1/6,t,s]/(s)$$ 
	therefore it is enough to prove that $\rho_4^*(H)$ restricted to this open is of the form $-2t$ to have the surjectivity of $\rho_4^*$. 
	
	We know that $\Mtilde_{1,1}\simeq [\AA^2/\gm]$, therefore it is enough to restrict to $\cA_4\vert_{\cB\gm}$ because the pullback of the closed immersion $\cB\gm \into \Mtilde_{1,1}$ is an isomorphism of Chow rings. Similarly to the proof of \Cref{prop:rho-6-surj}, a simple computation shows that $\cA_4\vert_{\cB\gm}$ is isomorphic to $[\AA^1/\gm]$ where $\gm$ acts with weight $-2$. An element in $\lambda \in \AA^1$ is equivalent to the inclusion of algebras $\kappa[t]/(t^2) \into \kappa[t]/(t^4)$ defined by the association $t\mapsto t^2+\lambda t^3$.
	
	The locus $\Htilde_3$ coincides with the locus in $[\AA^1/\gm]$ described by the equation $\lambda=0$. Therefore $\rho_4^*(H)=-2s$ and we are done. 
\end{proof}

Before going to study the morphism $\rho_2$, we need to understand its source. We have that 
$$\cA_{2}\simeq \cA_{2}'\simeq \Mtilde_{2,1} \times [\AA^1/\gm].$$
Recall that $\Mtilde_{2,1}$ is an open substack of $\Ctilde_2$ thanks to \Cref{theo:contrac}. Therefore, the Chow ring of $\Mtilde_{2,1}$ is a quotient of the one of $\Ctilde_2$. 

\begin{lemma}\label{lem:chow-ring-C2}
	The Chow ring of $\Ctilde_2$ is a quotient of the polynomial ring generated by
	\begin{itemize}
		\item[$\bullet$] the $\lambda$-classes $\lambda_1$ and $\lambda_2$ of degree $1$ and $2$ respectively,
		\item[$\bullet$] the $\psi$-class $\psi_1$,
		\item[$\bullet$] two classes $\theta_1$ and $\theta_2$ of degree $1$ and $2$ respectively;
	\end{itemize}
	furthermore, the ideal of relations is generated by 
	\begin{itemize}
		\item[$\bullet$]  $\lambda_2-\theta_2-\psi_1(\lambda_1-\psi_1)$,
		\item[$\bullet$]  $\theta_1(\lambda_1+\theta_1)$,
		\item[$\bullet$]  $\theta_2\psi_1$,
		\item[$\bullet$]  $\theta_2(\lambda_1+\theta_1-\psi_1)$,
		\item[$\bullet$]  an homogeneous polynomial of degree $7$.
	\end{itemize}
\end{lemma}

\begin{proof}
	We do not describe all the computation in details. The idea is to use the stratification introduced in Section 4 of \cite{DiLorPerVis},i.e. 
	$$\ThTilde_2 \subset \ThTilde_1 \subset \Ctilde_2$$ 
	where $\ThTilde_1$ is the pullback of $\Detilde_1$ through to morphism $\Ctilde_2 \rightarrow \Mtilde_2$ and $\ThTilde_2$ is the closed substack of $\Ctilde_2$ parametrizing pairs $(C,p)$ such that $p$ is a separating node. We denote by $\theta_1$ and $\theta_2$ the fundamental classes of $\ThTilde_1$ and $\ThTilde_2$. Notice that the only difference with our situation is in the open stratum $\Ctilde_2 \setminus \ThTilde_1$. In fact, the authors of \cite{DiLorPerVis} proved in Proposition 4.1 that  
	$$\Ctilde_2^2\setminus \ThTilde_1 \simeq [U/B_2]$$
	where $U$ is an open inside a $B_2$-representation $\widetilde{\AA}(6)$.
	The same proof generalizes in the case $r=7$ (see Proposition 4.1 of \cite{Per3}) and it gives us that 
	$$ \Ctilde_2^7\setminus \ThTilde_1 \simeq [\widetilde{\AA}(6)\setminus 0/B_2]$$
	and therefore the zero section in $\widetilde{\AA}(6)$ gives us a relation of degree $7$.  We also have the following isomorphisms:
	\begin{itemize}
		\item $\ThTilde_1 \setminus \ThTilde_2 \simeq (\Ctilde_{1,1}\setminus \Mtilde_{1,1})\times \Mtilde_{1,1}$,
		\item $\ThTilde_2 \simeq \Detilde_1$;
	\end{itemize}
 thus we have the following descriptions of the Chow rings of the strata:
	\begin{itemize}
		\item $\ch(\Ctilde_2\setminus \ThTilde_1) \simeq \ZZ[1/6,t_0,t_1]/(f_7)$,
		\item $\ch(\ThTilde_1 \setminus \ThTilde_2
		) \simeq \ZZ[1/6,t,s]$,
		\item $\ch(\ThTilde_2) \simeq \ZZ[1/6,\lambda_1,\lambda_2]$
	\end{itemize}
where $f_7$ is an homogeneous polynomial of degree $7$. Finally, one can prove the following identities:
\begin{itemize}
	\item $\lambda_1\vert_{\Ctilde_2\setminus \ThTilde_1} = -t_0-t_1$, $\lambda_1\vert_{\ThTilde_1\setminus \ThTilde_2}=-t-s$;
	\item $\lambda_2\vert_{\Ctilde_2\setminus \ThTilde_1}=t_0t_1$, $\lambda_2\vert_{\ThTilde_1\setminus \ThTilde_2}=st$;
	\item $\psi_1\vert_{\Ctilde_2\setminus \ThTilde_1}=t_1$, $\psi_1\vert_{\ThTilde_1\setminus \ThTilde_2}=t$, $\psi_1\vert_{\ThTilde_2}=0$;
	\item $\theta_1\vert_{\ThTilde_1\setminus \ThTilde_2}=t+s$, $\theta_1\vert_{\ThTilde_2}=-\lambda_1$;
	\item $\theta_2\vert_{\ThTilde_2}=\lambda_2$.
\end{itemize}
The result follows from applying the gluing lemma.
\end{proof}

\begin{remark}
	Clearly, $\rho_2^*$ cannot be surjective at the level of Chow rings, as it not true even at the level of Picard groups. In fact, the Picard group of $\Mtilde_3$ is an abelian free group of rank $3$ while the Picard group of $\Mtilde_{2,1}\times [\AA^1/\gm]$ is an abelian free group of rank $4$. 
\end{remark}

We are ready for the proposition.
\begin{proposition}\label{prop:gener-rho-2}
	The image of the pushforward of
	$$\rho_2: \Mtilde_{2,1} \times [\AA^1/\gm] \simeq \cA_2 \longrightarrow \Mtilde_3\setminus \widetilde{\cA}_{\geq 3}$$
	is generated by the elements $\rho_{2,*}(1)$, $\rho_{2,*}(s)$ and $\rho_{2,*}(s\theta_1)$, where $s$ is the generator of the Chow ring of $[\AA^1/\gm]$ and $\theta_1$ is the fundamental class of the locus parametrizing curves with a separating node.
\end{proposition} 

\begin{proof}
	For this proof, we denote by $\lambda_1$ and $\lambda_2$ the Chern classes of the Hodge bundle of $\Mtilde_{2,1}$, whereas the $i$-th Chern class of the Hodge bundle of $\Mtilde_3$ is denoted by $c_i(\HH)$ for $i=1,2,3$.
	
	We need to describe the pullback of the generators of the Chow ring of $\Mtilde_3$ through $\rho_2$. By construction, it is easy to see that $\rho_2^*(\delta_1)=s+\theta_1$, $\rho_2^*(\delta_{1,1})=\theta_2+s\theta_1$ and $\rho_2^*(\delta_{1,1,1})=s\theta_2$. 
	
	Notice that $\rho_2^{-1}(\Htilde_3)$ is the fundamental class of the closed substack $\Mtilde_{2,\omega} \times [\AA^1/\gm]$ of $\cA_2$, where $\Mtilde_{2,\omega}$ is the closed substack of $\Mtilde_{2,1}$ which parametrizes pairs $(C,p)$ such that $p$ is fixed by the (unique) involution of $C$. To compute its class, we need to use the stratification used in the proof of \Cref{lem:chow-ring-C2}.
	In the open stratum $\Mtilde_{2,1}\setminus \ThTilde_1$, Proposition 4.4 of \cite{Per3} implies that the restriction of $[\Mtilde_{2,\omega}]$ is equal to $-\lambda_1-3 \psi_1$.
	In the stratum $\ThTilde_1\setminus \ThTilde_2$, we have that the restriction is of the form $-3\psi_1$. This implies that $\rho_2^*(H)=-\lambda_1-3\psi_1-\theta_1$.
	
	Finally, to compute the restriction of $c_i(\HH)$ for $i=1,2,3$, we can restrict to the closed substack $\Mtilde_{2,1}\times \cB\gm \into \Mtilde_{2,1} \times [\AA^1/\gm]$ as the pullback of the closed immersion is clearly an isomorphism because it is the zero section of a vector bundle. The explicit description of the isomorphism $\Mtilde_{2,1}\times [\AA^1/\gm] \simeq \cA_2$ (which was constructed in Section 2 of \cite{DiLorPerVis}) implies that the morphism $\rho_2\vert_{\Mtilde_{2,1}\times \cB\gm}$ maps an object $(\widetilde{C}/S,q)$ to the object $(C/S,p)$ in the following way: consider the projective bundle $\PP(N_{q}\oplus N_0)$ over $S$, where $N_q$ is the normal bundle of the section $q$ and $N_0$ is the pullback to $S$ of the $1$-dimensional representation of $\gm$ of weight $1$; we have two natural sections defined by the two subbundles $N_q$ and $N_0$ of $N_q\oplus N_0$, namely $\infty$ and $0$; the object $(C/S,p)$ is defined by gluing $\infty$ with $q$, pinching in $0$ and then setting $p:=0$. A computation identical to the one of Proposition 5.9 of \cite{DiLorPerVis} implies the following formulas:
	\begin{itemize}
		\item $\rho_2^*(c_1(\HH))=\lambda_1+\psi_1-s$,
		\item $\rho_2^*(c_2(\HH))=\lambda_2+\lambda_1(\psi_1-s)$,
		\item $\rho_2^*(c_3(\HH))=\lambda_2(\psi_1-s)$.
	\end{itemize}
	The description of the restrictions of the generators of $\ch(\Mtilde_3)$ gives us that the image of $\rho_{2,*}$ is the ideal generated by $\rho_{2,*}(s^i)$ for every $i$ non-negative integer. Moreover, we have that 
	$$\rho_2^*(\delta_{1,1,1})=s\theta_2=s(\rho_2^*(\delta_{1,1})-s(\rho_2^*(\delta_1)-s))$$
	which implies that $\rho_{2,*}(s^i)$ is in the ideal generated by $\rho_{2,*}(1)$, $\rho_{2,*}(s)$ and $\rho_2^*(s^2)$ for every $i\geq 3$. Finally, we have that 
	$$s\theta_1=s(\rho_2^*(\delta_1)-s)$$
	therefore we can use $s\theta_1$ as a generator with $\rho_{2,*}(1)$ and $\rho_{2,*}(s)$ instead of $\rho_{2,*}(s^2)$.
\end{proof}

\begin{remark}
	Notice that $\rho_{2,*}(s)$ is equal to the fundamental class of the image of the morphism $\Mtilde_{2,1} \times \cB\gm \rightarrow \Detilde_1 \into \Mtilde_{3}$.  We denote this closed substack $\Detilde_1^c$; it parametrizes curves $C$ obtained by gluing a genus $2$ curve with a genus $1$ cuspidal curve in a separating node.
	
	In the same way, $\rho_{2,*}(s\theta_1)$ is equal to the fundamental class of the image of the morphism $\ThTilde_1 \times \cB\gm \into \Detilde_{1,1} \into \Mtilde_3$. We denote this closed substack $\Detilde_{1,1}^c$; it parametrizes curves $C$ in $\Detilde_{1,1}$ such that one of the two elliptic tails is cuspidal.
\end{remark}

\subsection*{Generators for the image of $\rho_{n,*}$ if $n$ is odd}

Now we deal with the odd case. This is a bit more convoluted as we have several strata to deal with for every $n$. Let us recall the descriptions we have. See \Cref{prop:descr-an-disp-ns} and \Cref{prop:descr-an-disp-i}.

First of all, $\cA_{2m-1}$ is the disjoint union of $\cA_{2m-1}^{\rm ns}$ and $\cA_{2m-1}^i$ for $0\leq i\leq (g-m+1)/2$. Because $g=3$, we have the following possibilities:
\begin{itemize}
	\item if $m=4$, we have only one component, namely $\cA_7^0$;
	\item if $m=3$, we have two components, namely $\cA_5^0$ and $\cA_5^{\rm ns}$;
	\item if $m=2$, we have three components, namely $\cA_5^0$, $\cA_5^1$ and $\cA_5^{\rm ns}$.
\end{itemize}
First of all, notice that $\cA_5^0$ is empty, due to the stability condition. Therefore we need to deal with $5$ components. 

We start with the case $m=4$.

\begin{proposition}\label{prop:rho-7-surj}
	The pullback of the morphism 
	$$\rho_7:\cA_7^0 \longrightarrow \Mtilde_3$$
	is surjective.
\end{proposition}

\begin{proof}
	The proof is similar to one of \Cref{prop:rho-6-surj}. First of all, we describe the Chow ring of $\cA_7^0$. We can apply \Cref{prop:descr-an-disp-i} and Lemma 2.5 of \cite{Per3} to get that 
	$$\ch(\cA_7^0) \simeq \ch(I_{7}^0)^{\rm inv}$$
 	where the invariants are taking with respect of the action of $C_2$ induced by the involution defined by the association $$\Big((C_1,p_1),(C_2,p_2),\phi\Big) \mapsto \Big((C_2,p_2),(C_1,p_1),\phi^{-1}\Big).$$
 	
 	By construction, we have that $I_7^0$ is the fiber product of the diagram
 	$$
 	\begin{tikzcd}
 		& {[E_{4,1}/G_4^{\times 2}]} \arrow[d] \\
 		\cB(\gm \ltimes \ga)^{\times 2} \arrow[r, "{\cB(f^{\times 2})}"] & \cB (G_4^{\times 2})                
 	\end{tikzcd}
 	$$
 	where the morphism $f$ is described in the proof of \Cref{prop:rho-6-surj} (see also \Cref{not:algebras}). A simple computation shows that 
 	$$ I_7^0 \simeq [\AA^1/\gm\ltimes \ga]$$
 	where $\ga$ acts trivially and $\AA^1$ is the $\gm$-representation with weight $2$. Furthermore, one can prove that $C_2$ acts trivially on $\gm\ltimes \ga$ and acts on $\AA^1$ by the rule $\lambda \mapsto -\lambda$. Therefore it is clear that
 	$$\ch(I_7^0) \simeq \ZZ[1/6,s]$$
 	where $s$ is the generator of the Chow ring of $\cB \gm$ and a simple computation shows the $\rho_7^{*}(H)=2s$.
\end{proof}

We now deal with the case $m=3$. We have to split it in two subcases, namely $\cA_5^0$ and $\cA_5^{\rm ns}$. We denote by $\rho_5^{0}$ and $\rho_5^{\rm ns}$ the restriction of $\rho_5$ to the two connected components $\cA_5^0$ and $\cA_5^{\rm ns}$ respectively.

\begin{proposition}\label{prop:rho-5-0-surj}
	The pullback of the morphism 
	$$\rho_5^{0}:\cA_5^0 \longrightarrow \Mtilde_3$$
	is surjective.
\end{proposition}

\begin{proof}
	In this case, \Cref{prop:descr-an-disp-i} tells us that $\cA_5^0$ is isomorphic to $I_5^0$. We have a commutative diagram
	$$
	\begin{tikzcd}
		& {[E_{3,1}/G_3\times G_3]} \arrow[d] \\
		{\cB(\gm \ltimes \ga)\times \Mtilde_{1,1}\times [\AA^1/\gm]} \arrow[r, "{(\cB f,g)}"] & \cB (G_3\times G_3)                
	\end{tikzcd}
	$$
	where the morphism $\cB f:\cB(\gm \ltimes \ga) \rightarrow \cB G_3$ is the same as in \Cref{prop:rho-7-surj}, whereas we recall that 
	$$ g:\Mtilde_{1,[3]} \longrightarrow G_3$$
	is defined by the association $(E,e)\mapsto \cO_{E}/m_e^3$. Because the morphism $f$ is injective (and therefore $\cB f$ is representable) and $E_{3,1}$ is a $\gm \ltimes \ga$-torsor, it is easy to verify that $\cA_{5}^0 \simeq \Mtilde_{1,1}\times [\AA^1/\gm]$. Therefore we have an isomorphism 
	$$\ch(\cA_5^0) \simeq \ZZ[1/6,t,s]$$
	where $t$ (respectively $s$) is the generator of the Chow ring of $\Mtilde_{1,1}$ (respectively $[\AA^1/\gm]$). We can describe $\rho_5^0$ in the following way: if we have a geometric point $(E,e)$ in $\Mtilde_{1,[3]}$, the image is the genus $3$ curve $(C,p)$ obtained by gluing the projective line $\PP^1$ to $E$ identifying $\infty$ with $e$ in a $A_5$-singularity (using the pushout construction as in Proposition 4.18 of \cite{Per1}) and setting $p=e$. Notice that there is a unique way of creating the $A_5$-singularity up to a unique isomorphism of $(\PP^1,\infty)$.
	
	Clearly, $\rho_5^{0,*}(\delta_1)=s$. However, $\rho_5^{0,*}(H)=0$, therefore we need to understand the pullback of the Chern classes of the Hodge bundle. It is enough to prove that $\rho_5^{0,*}(\lambda_1)=-2t+ns$ for some integer $n$. Therefore we can restrict the computation to the open substack $\Mtilde_{1,1}\subset \Mtilde_{1,1}\times [\AA^1/\gm]$. Moreover, as the closed embedding $\cB\gm \into \Mtilde_{1,1}$ is a zero section of a vector bundle over $\cB\gm$, it is enough to do the computation restricting everything to $\cB\gm \into \Mtilde_{1,1}\times [\AA^1/\gm]$.  Therefore, suppose we have an elliptic cuspidal curve $(E,e)$ with $e$ a smooth point, the image through $\rho_5^0$ is a $1$-pointed genus $3$ curve $(C,p)$ constructed gluing the projective line $(\PP^1,\infty)$ and $(E,e)$ (identifying $e$ and $\infty$) in an $A_5$-singularity and setting $p:=e$. We need to understand the $\gm$-action over the vector space $\H^0(C,\omega_C)$. 
	Consider the exact sequence
	$$ \begin{tikzcd}
		0 \arrow[r] & \cO_C \arrow[r] & \cO_{E}\oplus \cO_{\PP^1} \arrow[r] & Q \arrow[r] & 0
	\end{tikzcd}$$
	and the induced long exact sequence on the global sections
	$$
	\begin{tikzcd}
		0 \arrow[r] & \kappa \arrow[r] & \kappa^{\oplus 2} \arrow[r] & Q \arrow[r] & {\H^1(C,\cO_C)} \arrow[r] & {\H^1(E,\cO_E)} \arrow[r] & 0.
	\end{tikzcd}
    $$
    We know that $\lambda_1=-c_1^{\gm}(\H^1(C,\cO_C))$ and $c_1^{\gm}(\H^1(E,\cO_E))=t$. It is enough to describe $c_1^{\gm}(Q)$. Recall that $Q$ fits in a exact sequence of $\gm$-representations
    $$
    \begin{tikzcd}
    	0 \arrow[r] & \cO_{3\infty}=\kappa[t]/(t^3) \arrow[r] & \cO_{3e}\oplus \cO_{3\infty}=\kappa[t]/(t^3)^{\oplus 2}\arrow[r] & Q \arrow[r] & 0
    \end{tikzcd}
	$$ 
	where as usual $\cO_{3e}$ (respectively $\cO_{3\infty}$) is the quotient $\cO_E/m_e^3$ (respectively $\cO_{\PP^1}/m_{\infty}^3)$. Therefore an easy computation shows that $c_1(Q)=-2t$ and thus the restriction of $\lambda_1$ to $\cB\gm$ is equal to $-2t$.
\end{proof}

Finally, a proof similar to the one of \Cref{prop:rho-7-surj} gives us the following result.

\begin{proposition}
	The pullback of the morphism 
	$$\rho_5^{\rm ns}:\cA_5^{\rm ns} \longrightarrow \Mtilde_3$$
	is surjective.
\end{proposition}

\begin{proof}
	 We leave to the reader to check the details. See also \Cref{prop:rho-3-surj} for a similar result in the non-separating case.
\end{proof}

It remains to prove the case for $m=2$, or the strata classifying tacnodes.

\begin{proposition}\label{prop: rho-3-1-surj}
	The pullback of the morphism 
	$$ \rho_3^{1}: \cA_3^1 \longrightarrow  \Mtilde_3$$
	is surjective.
\end{proposition}

\begin{proof}
	Thanks to \Cref{prop:descr-an-disp-i}, we can describe $\cA_3^1$ using a $C_2$-action on $I_3^1$, which is a $\gm$-torsor over the stack $\Mtilde_{1,[2]}\times \Mtilde_{1,[2]}$ where $\Mtilde_{1,[2]}\simeq \Mtilde_{1,1}\times [\AA^1/\gm]$. In fact, Lemma 2.5 of \cite{Per3} implies that 
	$$ \ch(\cA_3^1)\simeq \ch(I_3^1)^{C_2-{\rm inv}}.$$
	Because the pullback of the closed immersion 
	$$ \Mtilde_{1,1} \times \cB\gm \into \Mtilde_{1,1} \times [\AA^1/\gm]$$ 
	is an isomorphism at the level of Chow rings, it is enough to understand the $\gm$-torsor when restricted to $(\Mtilde_{1,1}\times \cB\gm)^{\times 2}$. Let us denote by $t_i$ (respectively $s_i$) the generator of the Chow ring of $\Mtilde_{1,1}$ (respectively $\cB\gm$) seen as the $i$-th factor of the product for $i=1,2$. Exactly as we have done in \Cref{prop:gener-rho-2}, we can describe the objects in this product as  $$\Big((E_1,e_1), (\PP(N_{e_1}\oplus N_{s_1}),0,\infty), (E_2,e_2), (\PP(N_{e_2}\oplus N_{s_2}),0,\infty)\Big).$$ We recall that $N_{e_i}$ is the normal bundle of the section $e_i$ and $N_{s_i}$ is the representation of $\gm$ (whose generator is $s_i$) with weight $1$ (for $i=1,2$). By construction, the first Chern class of $N_{e}$ is the $\psi$-class associated to an object $(E,e)$ in $\Mtilde_{1,1}$. 
	
	The $\gm$-torsor comes from identifying the two tangent spaces at $\infty$ of the two projective bundles. A computation shows that its class (namely the first Chern class of the line bundle associated to it) in the Chow ring of the product is of the form $t_1-s_1-t_2+s_2$. Therefore, if we denote by $b_i:=t_i-s_i$ for $i=1,2$, we have the following  description of the Chow ring of $I_3^1$:
	$$\ch(I_3^1)\simeq \ZZ[1/6,t_1,t_2,b_1,b_2]/(b_1-b_2).$$
	Furthermore, the $C_2$-action over $I_3^1$ translates into an action on the Chow ring  defined by the association 
	$$ (t_1,t_2,b_1,b_2) \mapsto (t_2,t_1,b_2,b_1)$$
	therefore it is easy to compute the ring of invariants. We have the following result:
	$$ \ch(\cA_3^1)\simeq \ZZ[1/6,d_1,d_2,b]$$
	where $d_1:=t_1+t_2$ and $d_2:=t_1t_2$. An easy computation shows that $\rho_{3}^{1,*}(\delta_1)=d_1-2b$ and $\rho_3^{1,*}(\delta_{1,1})=d_2-bd_1+b^2$. Finally, a computation identical to the one in the proof of \Cref{prop:rho-5-0-surj} for the $\lambda$-classes, shows us that $\rho_{3}^{1,*}(\lambda_1)=b-d_1$. The statement follows.
\end{proof}

Finally, we arrived at the end of this sequence of abstract computations.

\begin{proposition}\label{prop:rho-3-surj}
	The pullback of the morphism 
	$$\rho_3^{\rm ns}: \cA_3^{\rm ns} \longrightarrow \Mtilde_3$$
	is surjective.
\end{proposition}

\begin{proof}
	For this proof, we denote by $\lambda_i$ the Chern classes of the Hodge bundle of $\Mtilde_{1,2}$, while the Chern classes of the Hodge bundle of $\Mtilde_3$ is denoted by $c_i(\HH)$.
	
	Thanks to \Cref{prop:descr-an-disp-ns}, we have that the Chow ring of $\cA_3^{\rm ns}$ can be recovered by the Chow ring of the $\gm$-torsor $I_3^{\rm ns}$ over $\Mtilde_{1,2}\times [\AA^1/\gm]\times [\AA^1/\gm]$.
	
	First of all, we know that $\Mtilde_{1,2}\simeq \Ctilde_{1,1}$ thanks to \Cref{theo:contrac} and we know the Chow ring of $\Ctilde_{1,1}$ is isomorphic to 
	$$\ZZ[1/6,\lambda_1,\mu_1]/(\mu_1(\lambda_1+\mu_1));$$
	see Proposition 3.3 of \cite{DiLorPerVis}. It is important to remark that $\mu_1$ is the fundamental class of the locus in $\Ctilde_{1,1}$ parametrizing $(E,e_1,e_2)$ such that the two sections coincide.
	
	 We need to understand the class of the $\gm$-torsor $I_3^{\rm ns}$ over $\Mtilde_{1,2}\times [\AA^1/\gm] \times [\AA^1/\gm]$. As usual, we reduce to the closed substack $\Mtilde_{1,2} \times \cB\gm \times \cB\gm$.  If we denote by $s_1$ (respectively $s_2$) the generator of the Chow ring of the first $\cB\gm$ (respectively the second $\cB\gm$) in the product, we have that the same description we used in \Cref{prop: rho-3-1-surj} for the $\gm$-torsor applies here and therefore we only need to understand the description of the two $\psi$-classes in $\ch(\Mtilde_{1,2})$, namely $\psi_1$ and $\psi_2$. We claim that $\psi_1=\psi_2$. Consider the autoequivalence 
	$$ \Mtilde_{1,2} \longrightarrow \Mtilde_{1,2},$$ 
	which is defined by the association $(E,e_1,e_2) \mapsto (E,e_2,e_1)$ (and therefore acts on the Chow rings sending $\psi_1$ in $\psi_2$ and viceversa). It is easy to see that is isomorphic to the identity functor because of the unicity of the involution.
	
	This implies that the class associated to the torsor $I_3^{\rm ns}$ is of the form $s_1-s_2$. Because the action of $C_2$ on $I_3^1$ translates into the involution 
	$$ (\lambda_1,\mu_1,s_1,s_2) \mapsto (\lambda_1,\mu_1,s_2,s_1)$$
	of the Chow ring, we finally have
	$$ \ch(\cA_3^{\rm ns})\simeq \ZZ[1/6, \lambda_1,\mu_1, s]/(\mu_1(\lambda_1+\mu_1))$$
	where $s:=s_1=s_2$. It is easy to see that $\rho_3^{\rm ns,*}(\delta_1)=\mu_1$. Moreover, the same ideas for the computations of the $\lambda$-classes used in \Cref{prop:rho-5-0-surj} gives us that $\rho_3^{\rm ns,*}(c_1(\HH))=-s$. Finally, it is enough to prove that $\rho_3^{\rm ns,*}(H)=-12\lambda_1$ modulo the ideal $(\mu_1,s)$, therefore we can restrict our computation to $\Mtilde_{1,2}\setminus \Mtilde_{1,1}\subset \Mtilde_{1,2}\times [\AA^1/\gm] \times [\AA^1/\gm]$. Notice that in this situation the $\gm$-torsor is trivial. Recall that we have the formula $H=9c_1(\HH)-\delta_0-3\delta_1$ by \cite{Est}. Therefore it follows that $\rho_3^{\rm ns,*}(H)=-\delta_0$. To compute $\delta_0$ in $\Mtilde_{1,2}\setminus \Mtilde_{1,1}$, we can consider the natural morphism 
	$$ \Ctilde_{1,1} \longrightarrow \Mtilde_{1,1}$$
    and use the fact that $\delta_0 = 12\lambda_1$ in $\ch(\Mtilde_{1,1})$.
\end{proof}

\Cref{lem:strata} implies that the image of $\rho_{\geq 2,*}$ in $\ch(\Mtilde_3)$ is generated by the following cycles:
\begin{itemize}
	\item[$\bullet$] the fundamental classes of $\Detilde_1^c$ and $\Detilde_{1,1}^c$;
	\item[$\bullet$] the fundamental classes of the images of $\rho_7$, $\rho_5^0$ and $\rho_3^1$, which are closed inside $\Mtilde_3$ because of Lemma 4.7 of \cite{Per1}; by abuse of notation, we denote these closed substack by $\cA_7$, $\cA_5^0$ and $\cA_3^1$ respectively;
	\item[$\bullet$] the fundamental classes of the closure of the images of $\rho_6$, $\rho_5^{\rm ns}$, $\rho_4$, $\rho_3^{\rm ns}$ and $\rho_2$; by abuse of notation, we denote these closed substack as $\cA_6$, $\cA_5$, $\cA_4$, $\cA_3$ and $\cA_2$ respectively.
\end{itemize}
 
\begin{remark}
	Notice that $\cA_6$, $\cA_4$ and $\cA_2$ are the stacks we previously denoted by $\widetilde{\cA}_{\geq 6}$, $\widetilde{\cA}_{\geq 4}$ and $\widetilde{\cA}_{\geq 2}$ respectively. Moreover, the stacks $\cA_5^{\rm ns}$ and $\cA_3^{\rm ns}$ are substacks of $\widetilde{\cA}_{\geq 5}$ and $\widetilde{\cA}_{\geq 3}$ respectively.
\end{remark}

\begin{corollary}\label{cor:relations}
	The Chow ring of $\Mbar_3$ is the quotient of the Chow ring of $\Mtilde_3$ by the fundamental classes of $\cA_7$, $\cA_6$, $\cA_5^0$, $\cA_5^{\rm ns}$, $\cA_4$, $\cA_3^1$, $\cA_3^{\rm ns}$, $\cA_2$, $\Detilde_1^c$ and $\Detilde_{1,1}^c$.
\end{corollary}

\section{Explicit description of the relations}\label{sec:3}

We illustrate the strategy to compute the explicit description of the relations listed in \Cref{cor:relations}. Suppose we want to compute the fundamental class of a closed substack $X$ of $\Mtilde_3$. First of all, we need to compute the classes of the restrictions of $X$ on every stratum, namely 
$$ X\vert_{\Mtilde_3\setminus (\Htilde_3 \cup \Detilde_1)}, X\vert_{\Htilde_3\setminus \Detilde_1}, X\vert_{\Detilde_1\setminus \Detilde_{1,1}}, X\vert_{\Detilde_{1,1}\setminus \Detilde_{1,1,1}}, X\vert_{\Detilde_{1,1,1}}.$$
Once we have the explicit descriptions, we need to patch them together. Let us show how to do it for the first two strata, i.e. $\Mtilde_3\setminus(\Htilde_3 \cup \Detilde_1)$ and $\Htilde_3 \setminus \Detilde_1$. Suppose we have the description of $X^q:=X\vert_{\Mtilde_3 \setminus (\Detilde_{1}\cup \Htilde_3)}$ and of $X^h:=X\vert_{\Htilde_3 \setminus \Detilde_1}$ in their respective Chow rings. Then we can compute $X\vert_{\Mtilde_{3}\setminus \Detilde_1}$ using \Cref{lem:gluing}. Suppose we are given
$$ X\vert_{\Mtilde_3\setminus \Detilde_1}=p + Hq  \in \ch(\Mtilde_3\setminus \Detilde_1)$$
an expression of $X$, we need to compute the two polynomials $p$ and $q$. If we restrict $X$ to $\Mtilde_3\setminus (\Htilde_3\cup \Detilde_1)$, we get that the polynomial $p$ can be chosen to be just any lifting of $X^q$. Now if we restrict to $\Htilde_3 \setminus \Detilde_1$ we get that
$$ i_h^*p + c_{\rm top}(N_{\cH|\cM})q = X^h$$
where as usual $i_h:\Htilde_3 \setminus \Detilde_1 \into \Mtilde_3 \setminus \Detilde_1$ is the closed immersion of the hyperelliptic stratum and $N_{\cH|\cM}$ is the normal bundle of this immersion. Because of the commutativity of the diagram in \Cref{lem:gluing}, we have that $X^h-i_h^*p$ is in the ideal in $\ch(\Htilde_3 \setminus \Detilde_1)$ generated by $c_{\rm top}(N_{\cH|\cM})$. However the top Chern class is a non-zero divisor, thus we can choose $q$ to be just a lifting of $\widetilde{q}$, where $\widetilde{q}$ is an element in $\ch(\Htilde_3 \setminus \Detilde_1)$ such that $X^h-i_h^*p=c_{\rm top}(N_{\cH|\cM})\widetilde{q}$. Although we have done a lot of choices, it is easy to see that the presentation of $X\vert_{\Mtilde_3 \setminus \Detilde_1}$ is unique in the Chow ring of $\Mtilde_3\setminus \Detilde_1$, i.e. two different presentations differ by a relation in the Chow ring.

We show how to apply this strategy firstly for the computations of the fundamental classes $\delta_{1}^c$ and $\delta_{1,1}^c$ of $\Detilde_1^c$ and $\Detilde_{1,1}^c$ respectively.
\begin{proposition}
	We have the following description:
	$$ \delta_1^c = 6\big(\delta_1(H+\lambda_1+3\delta_1)^2+4\delta_{1,1}(\lambda_1-H-2\delta_1)+12\delta_{1,1,1}\big)$$
	and 
	$$ \delta_{1,1}^c = 24(\delta_{1,1}(\delta_1+\lambda_1)^2+\delta_{1,1,1}\delta_1)$$
	in the Chow ring of $\Mtilde_3$. 
\end{proposition}
\begin{proof}
First of all, we have that $\Detilde_1^c\subset \Detilde_1$, therefore the generic expression of the class $\delta_1^c$ is of the form 
$$ \delta_1 p_2 + \delta_{1,1} p_1 + \delta_{1,1,1} p_0$$ 
where $p_0,p_1,p_2$ are homogeneous polynomial in $\ZZ[1/6,\lambda_1,\lambda_2,\lambda_3, \delta_1, \delta_{1,1},\delta_{1,1,1}, H]$ of degree respectively $0,1,2$. 
We start by restricting $\delta_1^c$ to $\Mtilde_3\setminus \Detilde_{1,1}$. Here we have the sequence of embeddings 
$$ \Detilde_1^c\cap (\Detilde_1 \setminus \Detilde_{1,1}) \into (\Detilde_1 \setminus \Detilde_{1,1}) \into \Mtilde_3 \setminus \Detilde_{1,1}$$ 
which implies that $\delta_1^c$ restricted to $\ch(\Detilde_{1}\setminus \Detilde_{1,1})$ is equal $24t^2(t+t_1)$, where $24t^2$ is the fundamental class of the closed embedding $\Detilde_1^c \into \Detilde_{1}$ (restricted to the open $\Detilde_1 \setminus \Detilde_{1,1}$) while $(t+t_1)$ is the normal bundle of the closed embedding $i_{1}:\Detilde_1 \setminus \Detilde_{1,1} \into \Mtilde_3 \setminus \Detilde_{1,1}$. Because $t+t_1$ is not a zero divisor in the Chow ring, we have that $i_1^*(p_2)=24t^2$ which implies $p_2=6(H+\lambda_1+3\delta_1)^2$ (see Proposition 4.4 of \cite{Per3}).

Now we have to compute the restriction of $\delta_1^c$ to $\Detilde_{1,1}\setminus \Detilde_{1,1,1}$. This is not trivial, because $\Detilde_1^c$ is contained in $\Detilde_1$ but the closed immersion $\Detilde_{1,1} \setminus \Detilde_{1,1,1} \into \Detilde_{1}\setminus \Detilde_{1,1,1}$ is not regular. As a matter of fact, one can prove that doing the naive computation does not work, i.e. the difference $\delta_1^c-\delta_1p_2$ restricted to $\Detilde_{1,1}\setminus \Detilde_{1,1,1}$ is not divisible by the top Chern class of the normal bundle of $i_{1,1}:\Detilde_{1,1}\setminus \Detilde_{1,1,1}\into \Mtilde_3 \setminus \Detilde_{1,1,1}$.  To do it properly, we can consider the following cartesian diagram
$$
\begin{tikzcd}
	{(\Mtilde_{2,1}\setminus \ThTilde_2)\times \cB\gm} \arrow[r, "\rho_2^c"]                                                          & {\Mtilde_3 \setminus \Detilde_{1,1,1}}                                  \\
	{(\Mtilde_{1,2}\setminus \Mtilde_{1,1})\times \Mtilde_{1,1}\times \cB\gm} \arrow[u, hook, "\mu_1\times \id"] \arrow[r, "\rho_{1,1}"] & {\Detilde_{1,1}\setminus \Detilde_{1,1,1}} \arrow[u, "{i_{1,1}}", hook]
\end{tikzcd}
$$
where 
$$\mu_1:(\Mtilde_{1,2}\setminus \Mtilde_{1,1})\times \Mtilde_{1,1}\simeq (\ThTilde_1\setminus \ThTilde_2) \into \Ctilde_2\setminus \ThTilde_2 \simeq \Mtilde_{2,1} \setminus \ThTilde_2$$
is described in the proof of \Cref{lem:chow-ring-C2} and $\rho_2^c$ is the restriction of $\rho_2$ to the closed substack $\Detilde_1^c$. Notice that $\mu_1\times \id$ is a regular embedding of codimension $1$ whereas $i_{1,1}$ is regular of codimension $2$. Excess intersection theory implies that 
$$ \delta_1^c=\rho_{2,*}^c(1)=\rho_{(1,1),*}(c_1(\rho_{1,1}^*N_{i_{1,1}}/N_{\mu_1\times \id}));$$
the normal bundle $N_{\mu_1}$ was described in Proposition 4.9 of \cite{DiLorPerVis} while $N_{i_{1,1}}$ was described in Proposition 5.6 of \cite{Per3} (see also the proof of \Cref{lem:chow-ring-C2}). A computation shows that 
$$ c_1(\rho_{1,1}^*N_{i_{1,1}}/N_{\mu\times \id})=(t+t_2)$$
where $t$ is the generator of $\ch(\Mtilde_{1,2}\setminus \Mtilde_{1,1})$ while $t_2$ is the generator of $\ch(\cB\gm)$. Finally, we need to compute $\rho_{(1,1),*}(t+t_2)$. This can be done by noticing that $\rho_{1,1}$ factors through the morphism described in Lemma 5.1 of \cite{Per3}, i.e. the diagram 
$$ 
\begin{tikzcd}
	{(\Mtilde_{1,2}\setminus \Mtilde_{1,1})\times \Mtilde_{1,1}\times \cB\gm} \arrow[r, "\tilde{\rho}_{1,1}"] \arrow[rd, "\rho_{1,1}"] & {(\Mtilde_{1,2}\setminus \Mtilde_{1,1})\times \Mtilde_{1,1} \times \Mtilde_{1,1}} \arrow[d, "\pi_2"] \\
	& {\Detilde_{1,1}\setminus \Detilde_{1,1,1}}                                                 
\end{tikzcd}
$$
is commutative, where $\tilde{\rho}_{1,1}$ is induced by the zero section $\cB\gm \into [\AA^2/\gm]\simeq \Mtilde_{1,1}$. A simple computation using the fact that $\pi_2$ is a $C_2$-torsor gives us $$\delta_1^c\vert_{\Detilde_{1,1}\setminus \Detilde_{1,1,1}}=24(c_1t^2-2c_2t+c_1^3-3c_1c_2).$$
We can now divide $(\delta_1^c-\delta_1p_2)\vert_{\Detilde_{1,1}\setminus \Detilde_{1,1,1}}$ by $c_2+tc_1+t^2$, which is the $2$nd Chern class of the normal bundle of $i_{1,1}$, and get $i_{1,1}^*p_1=24(-2t-3c_1)$. Therefore we have $p_1=24(\lambda_1-H-2\delta_1)$.

Finally, we have to restrict $\delta_1^c$ to $\Detilde_{1,1,1}$. Again we can consider the following diagram
$$
\begin{tikzcd}
	{\Mtilde_{1,1}\times \Mtilde_{1,1}\times \cB\gm} \arrow[r, "c_2\times \id"] \arrow[d, "\tilde{\rho}_2", hook] & \ThTilde_2 \times \cB\gm \arrow[r, "\mu_2\times \id", hook] \arrow[d, "\alpha"] & {\Mtilde_{2,1} \times \cB\gm} \arrow[d] \\
	{\Mtilde_{1,1}\times \Mtilde_{1,1}\times \Mtilde_{1,1}} \arrow[r, "c_6"]                            & {\Detilde_{1,1,1}} \arrow[r, "{i_{1,1,1}}", hook]                     & \Mtilde_3                              
\end{tikzcd}
$$
where the right square is cartesian, the morphism $\mu_2:\ThTilde_2 \into \Ctilde_2\simeq \Mtilde_{2,1}$ is the closed immersion described in Section 4 of \cite{DiLorPerVis} (see also the proof of \Cref{lem:chow-ring-C2}), the morphism $c_2:\Mtilde_{1,1}\times \Mtilde_{1,1}\rightarrow \Detilde_1\simeq \ThTilde_2$ is the gluing morphism (see proof of Lemma 4.8 of \cite{DiLorPerVis}) and $c_6$ is the morphism described in Lemma 6.1 of \cite{Per3}. 
Excess intersection theory implies that 
$$\delta_1^c\vert_{\Detilde_{1,1,1}}= \alpha_*(N_{\mu_2});$$ 
moreover we have that
$$N_{\mu_2}= \frac{(c_2\times \id)_*\tilde{\rho}_2^*(t_1+t_2)}{2}$$ 
where $t_i$ is the generator of the $i$-th factor of the product $\Mtilde_{1,1}^{\times 3}$ for $i=1,2,3$. Therefore
$$\delta_{1}^c\vert_{\Detilde_{1,1,1}}=c_{6,*}(12t_3^2(t_1+t_2))=24(c_1c_2-3c_3)$$
in the Chow ring of $\Detilde_{1,1,1}$.
This leads finally to the description. The same procedure can be used to compute the class of $\delta_{1,1}^c$.
	
\end{proof}

\begin{remark}
The relation $\delta_1^c$ gives us that we do not need the generator $\delta_{1,1,1}$  in $\ch(\Mbar_3)$.
\end{remark}

\subsection*{The fundamental class of $\cA_5^0$ and $\cA_3^1$}

Now we concentrate on describing two of the strata of separating singularities, namely $\cA_{3}^1$ and $\cA_5^1$. Let us start with $\cA_3^1$.

\begin{proposition}
	We have the following description 
	$$ [\cA_3^1]=\frac{1}{2}(H+\lambda_1+\delta_1)(3H+\lambda_1+\delta_1) $$
	in the Chow ring of $\Mtilde_3$.
\end{proposition}

\begin{proof}
	 A generic object in $\cA_3^1$ can be described as two genus $1$ curves intersecting in a separating tacnode. We claim that the morphism $\cA_3^1 \rightarrow \Mtilde_3$ (which is proper thanks to Lemma 4.7 of \cite{Per1}) factors through $\Xi_1\subset\Htilde_3 \into \Mtilde_3$. In fact, given an element $((E_1,e_1),(E_2,e_2),\phi)$ in $\cA_3^1$ with $E_1$ and $E_2$ smooth genus $1$ curves, we can consider the hyperelliptic involution of $E_1$ (respectively $E_2$) induced by the complete linear sistem of $\cO(2e_1)$ (respectively $\cO(2e_2)$). It is easy to see that their differentials commute with the isomorphism $\phi$ and therefore we get an involution of the genus $3$ curve whose quotient is a genus $0$ curve with one node. Because we have considered a generic element and the hyperelliptic locus is closed, we get the claim. In particular, $\cA_3^1\vert_{\Mtilde_3 \setminus (\Htilde_3 \cup \Detilde_1)}$ is zero. 
	
	Consider now the description of $\Htilde_3\setminus \Detilde_1$ as in Section 2 of \cite{Per3}. Because $\cA_3^1\setminus \Detilde_1 \into \Xi_1\setminus \Detilde_1 \into \Htilde_3 \setminus \Detilde_1$, we have that excess intersection theory gives us the equalities
	$$ [\cA_3^1]\vert_{\Htilde_3 \setminus \Detilde_1} = c_1(N_{\cH|\cM})[a_3]=\frac{2\xi_1-\lambda_1}{3}[a_3]$$ 
	where $a_3$ is the class of $\cA_3^1$ as a codimension $2$ closed substack in $\Htilde_3\setminus \Detilde_1$. Suppose we have an element $(Z/S,L,f)$ in $\Htilde_3 \setminus \Detilde_1$; first of all we have that $\cA_3^1\subset \Xi_1$, therefore $Z/S \in \cM_0^{1}$. Furthermore, because we have a separating tacnode between two genus $1$ curve, it is clear that the nodal section $n$ in $Z$ has to be in the branching locus, or equivalently $f(n)=0$. Due to the description of $\Xi_1$, we have that 
	$$ a_3 = (-2s)[\Xi_1] = (-2s)(-c_1) $$
	and therefore using $\lambda$-classes as generators
	$$ [\cA_3^1]\vert_{\Htilde_3 \setminus \Detilde_1}=\frac{1}{3}a_3 (2\xi_1-\lambda_1) = \frac{2}{9}(\xi_1+\lambda_1)(2\xi_1-\lambda_1)\xi_1.$$ 
	
	Let us focus on the restriction to $\Detilde_1 \setminus \Detilde_{1,1} \simeq (\Mtilde_{2,1}\setminus \ThTilde_1) \times \Mtilde_{1,1}$. It is easy to see that the only geometric objects that are in $\cA_3^1$ are of the form $((C,p),(E,e))$ where $p$ lies in an almost-bridge of the genus $2$ curve $p$, i.e. $p$ is a smooth point in a projective line that intersects the rest of the curve in a tacnode. Recall that we have an open immersion
	$$\Mtilde_{2,1}\setminus \ThTilde_1 \simeq \Ctilde_{2}\setminus \ThTilde_1$$
	described in \Cref{theo:contrac}. Through this identification, $(C,p)$ corresponds to a pair $(C',p')$ where $C'$ is an ($A_r$-)stable genus $2$ curve and $p'$ is a cuspidal point. Therefore it is easy to see that in the notation of Corollary 4.2 of \cite{Per3}, we have that the fundamental class of the separating tacnodes is described by the equations $s=a_5=a_4=0$.  We get the following expression
	$$ [\cA_3^1]\vert_{\Detilde_{1}\setminus\Detilde_{1,1}} = -2t_1(t_0-2t_1)(t_0-3t_1) \in \ch(\Detilde_1 \setminus \Detilde_{1,1}).$$ 
	
	The same idea works for $\Detilde_{1,1}\setminus \Detilde_{1,1,1}$ (using Lemma 5.2 of \cite{Per3}) and gives us the following description 
	$$ [\cA_3^1]\vert_{\Detilde_{1,1}\setminus \Detilde_{1,1,1}}= -24t^3 \in \ch(\Detilde_{1,1} \setminus \Detilde_{1,1,1}).$$
	
	Finally, it is clear that $\cA_3^1 \cap \Detilde_{1,1,1} = \emptyset$. 
\end{proof}

Now we focus on the fundamental class of $\cA_5^0$. A generic object in $\cA_5^0$ can be described as a genus $1$ curve and a genus $0$ curve intersecting in a $A_5$-singularity. Notice that this implies that $\cA_5^0 \cap \Htilde_3=\emptyset$ because the only possible involution of an $A_5$-singularity has to exchange the two irreducible components. In the same way, it is easy to prove that $\cA_5^0 \cap \Detilde_{1,1} = \emptyset$. 

The intersection of $\cA_5^0$ with $\Detilde_{1}$ is clearly transversal, therefore
$$[\cA_5^0]\vert_{\Detilde_{1}\setminus \Detilde_{1,1}}=[\cA_5^0 \cap \Detilde_1].$$

\begin{lemma}\label{lem:a-5-0}
	We have the following equality
	$$ [\cA_5^0]\vert_{\Detilde_{1}\setminus \Detilde_{1,1}} = 72(t_0+t_1)^3t_0t_1 - 384(t_0+t_1)(t_0t_1)^2$$
	in the Chow ring of $\Detilde_1 \setminus \Detilde_{1,1}$.
\end{lemma}

\begin{proof}
	Because every $A_5$-singularity for a curve of genus $2$ is separating, it is enough to compute the fundamental class of the locus of $\cA_5$-singularities in $\Ctilde_2 \setminus \ThTilde_1$. We know that 
	$$\Ctilde_2 \setminus \ThTilde_1 \simeq [\widetilde{\AA}(6)\setminus 0/B_2]$$ 
	see Corollary 4.2 of \cite{Per3} for a more detailed discussion. We have that an element $(f,s) \in \widetilde{\AA}(6)$ defines a genus $2$ curve with an $A_5$-singularity if and only if $f \in \AA(6)$ has a root of multiplicity $6$. Because $B_2$ is a special group, it is enough to compute the $T$-equivariant fundamental class of the locus parametrizing sections of $\AA(6)$ which have a root of multiplicity $6$, where $T$ is the maximal torus inside $B_2$. Therefore one can use the formula in \Cref{rem:gener}.
\end{proof}

It remains to describe the restriction of $\cA_5^0$ in $\Mtilde_3\setminus (\Htilde_3 \cap \Detilde_1)$. First of all, we pass to the projective setting. Recall that 
$$ \Mtilde_3\setminus (\Htilde_3 \cap \Detilde_1) \simeq [U/\GL_3]$$  
where $U$ is an invariant open inside a $\GL_3$-representation of the space of forms in three coordinates of degree $4$. Because $U$ does not cointain the zero section, we can consider the projectivization $\overline{U}$ in $\PP^{14}$ and we have the isomorphism 
$$ \ch_{\GL_3}(U)\simeq \ch_{\GL_3}(\overline{U})/(c_1-h)$$ 
where $c_1$ is the first Chern class of the standard representation of $\GL_3$.  We want to describe the locus $X_5^0$ which parametrizes pairs $(f,p)$ such that $p$ is a separating $A_5$-singularity of $f$. Because a quartic in $\PP^2$ has at most one $A_5$-singularity, we can compute the pushforward through $\pi$ of the fundamental class of $X_5^0$ and then set $h=c_1$ to the get the fundamental class of $\cA_5^0$. Notice that pair $(f,p) \in X_5^0$ can be described as a cubic $g$ and a line $l$ such that $f=gl$ and $p$ is the only intersection between $g$ and $l$, or equivalently $p$ is a flex of $g$ and $l$ is the flex tangent. 

To describe $X_5^0$, we first introduce the locally closed substack $X_2$ of $[\PP^{14}\times \PP^2/\GL_3]$ which parametrizes pairs $(f,p)$ such that $p$ is a $A_r$-singularity of $f$ with $r\geq 2$ (eventually $r=\infty$).

Recall that we have the isomorphism  
$$ [\PP^{14}\times \PP^2/\GL_3] \simeq [\PP^{14}/H]$$
where $H$ is the stabilizer subgroup of $\GL_3$ of the point $[0:0:1]$ in $\PP^2$. The isomorphism is a consequence of the transitivity of the action of $\GL_3$ on $\PP^2$. Moreover, $H\simeq (\GL_2\times \gm)\ltimes \ga^2$; see Remark 3.13 of \cite{Per3}. We denote by $G:=(B_2\times \gm)\ltimes \ga^2$ where $B_2 \into \GL_2$ is the Borel subgroup of upper triangular matrices inside $\GL_2$.

Finally, let us denote by $L$ the standard representation of $\gm$, by $E$ the standard representation of $\GL_2$ and $E_0 \subset E$ the flag induced by the action of $B_2$ on $E$; $E_1$ is the quotient $E/E_0$. We denote by $V_3$ the $H$-representation
$$(L^{\vee} \otimes {\rm Sym}^3E^{\vee}) \oplus {\rm Sym}^4E^{\vee}.$$

\begin{lemma}\label{lem:X-2}
	In the setting above, we have the following isomorphism
	$$ X_2 \simeq [V_3 \otimes E_0^{\otimes 2} \otimes L^{\otimes 2}/G]$$
	of algebraic stacks.
\end{lemma}

\begin{proof}
Thanks to the isomorphism  
$$ [\PP^{14}\times \PP^2/\GL_3] \simeq [\PP^{14}/H]$$
we can describe $X_2$ as a substack of the right-hand side. The coordinates of $[\PP^{14}/H]$ are the coefficients of the generic polynomial $p=a_{00}+a_{10}x+ a_{01}y+ \dots +a_{04}y^4$ which is the dehomogenization of the generic quartic $f$ in the point $[0:0:1]$.
First of all, notice that $X_2$ is contained in the complement of the closed substack $X$ parametrizing polynomials $f$ such that its first and second derivatives in $x$ and $y$ vanish in $(0,0):=[0:0:1]$, thanks to Lemma 3.6 of \cite{Per3}. This is an $H$-equivariant subbundle of codimension $6$ in $[\PP^{14}/H]$ defined by the equations $$a_{00}=a_{10}=a_{01}=a_{20}=a_{11}=a_{02}=0.$$
We denote it simply by $[\PP^8/H]$. Moreover, $X_2$ is cointained in the locus parametrizing quartic $f$ such that $f$ is singular in $(0,0)$. This is an $H$-equivariant subbundle defined by the equations $a_{00}=a_{10}=a_{01}=0$. We denote it simply $[\PP^{11}/H]$, therefore $X_2 \into [\PP^{11}\setminus \PP^8/H]$. By construction, $\PP^{11}$ can be described as the projectivization of the $H$-representation $V$ defined as 
$$ (L^{\otimes -2}\otimes{\rm Sym}^2E^{\vee}) \oplus V_3$$
whereas $\PP^8$ is the projectivization of $V_3$.

Consider an element $p$ in $\PP^{11}$, described as the polynomial $a_{20}x^2+a_{11}xy+a_{02}y^2 + p_3(x,y) +p_4(x,y)$ where $p_3$ and $p_4$ are homogeneous polynomials in $x,y$ of degree respectively $3$ and $4$. Notice that $(0,0)$ is an $A_1$-singularity, i.e. an ordinary node, if and only if $a_{11}^2-4a_{02}a_{20}\neq 0$.  In fact, an $A$-singularity (eventually $A_{\infty}$) is a node if and only if it has two different tangent lines. Therefore $X_2$ is equal to the locus where the equality holds and therefore it is enough to describe $\VV(a_{11}^2-4a_{02}a_{20})$ in $[\PP^{11}\setminus \PP^8/H]$.

We define $W$ as the $H$-representation 
$$W:=(L^{\vee}\otimes E^{\vee}) \oplus (L^{\vee} \otimes {\rm Sym}^3E^{\vee}) \oplus {\rm Sym}^4E^{\vee}$$
and we consider the $H$-equivariant closed embedding of $W\into V$ induced by the morphism of $H$-schemes
$$ L^{\vee}\otimes E^{\vee} \into {\rm Sym}^2(L^{\vee} \otimes E^{\vee})\simeq L^{\otimes -2}\otimes {\rm Sym}^2E^{\vee}$$ 
which is defined by the association $(f_1,f_2) \mapsto (f_1^2,2f_1f_2,f_2^2)$. Now, consider the $\gm$-action on $W$ defined using weight $1$ on the $2$-dimensional vector space $(L^{\vee}\otimes E^{\vee})$ and weight $2$ on $V_3$. We denote by $\PP_{1,2}(W)$ the quotient stack and by $\PP_2(V_3)\into \PP_{1,2}(W)$ the closed substack induced by the embedding $V_3\into W$. One can prove that the morphism 
$$ [\PP_ {1,2}(W)\setminus \PP_2(V_3)/H] \into [\PP(V)\setminus \PP(V_3)/H]$$
induced by the closed immersion $W \into V$ is a closed immersion too and its stack-theoretic image is exactly the locus $\VV(a_{11}^2-4a_{02}a_{20})$. We are considering the action of $H$ on a stack as defined in \cite{Rom}. Finally, because the action of $\GL_2$ over $E^{\vee}$ is transitive, we have the isomorphism 
$$ [\PP_{1,2}(W) \setminus \PP_2(V_3)/H] \simeq [\PP_{1,2}(W_0) \setminus \PP_2(V_3)/G]$$ 
where the $G$-representation $W_0\into W$ is defined as $(L^{\vee}\otimes E_0^{\vee})\oplus V_3$. We want to stress that this is true only if we remove the locus $\PP_2(V_3)$, in fact the subgroup of stabilizers of $W_0$ in $W$ is equal to $H$ when restricted to $V_3$. Finally, we notice that we have an $H$-equivariant isomorphism $$\PP_{1,2}(W_0)\setminus \PP_2(V_3)\simeq V_3 \otimes (E_0^{\otimes 2}\otimes L^{\otimes 2})$$ of stacks.
	
\end{proof}

\begin{remark}
	By construction, an element $f \in V_3 \otimes E_0^{\otimes 2} \otimes L^{\otimes 2}$ is associated to the curve $y^2=f(x,y)$. Notice that $E^{\vee}$ is the vector space $E_0^{\vee}\oplus E_1^{\vee}$ where $x$ is a generator for $E_1^{\vee}$ and $y$ is a generator for $E_0^{\vee}$. Moreover,  $L^{\vee}$ is generated by $z$ where $(x,y,z)$ are a basis of the dual of the standard representation of $\GL_2$.
\end{remark}

\begin{corollary}
	We have an isomorphism of rings
	$$ \ch(X_2)\simeq \ZZ[1/6,t_1,t_2,t_3]$$ 
	where $t_1,t_2,t_3$ are the first Chern classes of the standard representations of the three copies of $\gm$ in $G$. Specifically, $t_1,t_2,t_3$ are the first Chern classes of $E_1,E_0,L$ respectively.
\end{corollary}

Recall that $X$ is the closed substack of $\PP^{14}\times \PP^2$ which parametrizes pairs $(f,p)$ such that $p$ is a singular point of $f$ but not an $A$-singularity, see Definition 3.4 of \cite{Per3}.
Thus, we have a closed immersion 
$$ i_2:=X_2 \into [(\PP^{14}\times \PP^2)\setminus X/\GL_3]$$ 
and we can describe its pullback at the level of Chow ring. We have an isomorphism
$$\ch_{\GL_3}(\PP^{14}\times \PP^2) \simeq \frac{\ZZ[1/6,c_1,c_2,c_3,h_{14},h_{2}]}{(p_2(h_2),p_{14}(h_{14}))}$$ 
where $c_i$ is the $i$-th Chern class of the standard representation of $\GL_3$, $h_{14}$ (respectively $h_{2}$) is the hyperplane section of $\PP^{14}$ (respectively of $\PP^{2}$) and $p_{14}$ (respectively $p_2$) is a polynomial of degree $15$ (respectively $3$) with coefficients in $\ch(\cB\GL_3)$. 
\begin{proposition}\label{prop:i-2}
	The closed immersion $i_2$ is the complete intersection in $[(\PP^{14}\times \PP^2)\setminus X/\GL_3]$ defined by equations 
	$$ a_{00}=a_{10}=a_{01}=a_{11}^2-4a_{20}a_{02}=0$$
	whose fundamental class is equal to
	$$2(h+k-c_1)(h+4k)((h+3k)^2-(c_1+k)(h+2k)+c_2).$$
	Moreover, the morphism $i_2^*$ is defined by the following associations:
\begin{itemize}
		\item $i_2^*(k)=-t_3$,
		\item $i_2^*(c_1)=t_1+t_2+t_3$, 
		\item $i_2^*(c_2)=t_1t_2+t_1t_3+t_2t_3$,
		\item $i_2^*(c_3)=t_1t_2t_3$,
		\item $i_2^*(h)=2(t_2+t_3)$,
\end{itemize}   
\end{proposition}

\begin{proof}
	It follows from the proof of \Cref{lem:X-2}.
\end{proof}

Because $i_2^*$ is clearly surjective, it is enough to compute the fundamental class of $X_5^0$ as a closed subscheme of $X_2$, choose a lifting through $i_2^*$ and then multiply it by the fundamental class of $X_2$.

We are finally ready to do the computation. 
\begin{corollary}
	The closed substack $X_5^0$ of $X_2$ is the complete intersection defined by the vanishing of the coefficients of $x^3, x^2y, x^4$ as coordinates of $V_3 \otimes E_0^{\otimes 2} \otimes L^{\otimes 2}$.
\end{corollary}

\begin{proof}
	The element of the representation $V_3 \otimes E_0^{\otimes 2} \otimes L^{\otimes 2}$ are the coefficients of polynomials of the form $p_3(x,y)+p_4(x,y)$ where $p_3$ (respectively $p_4$) is the homogeneous component of degree $3$ (respectively $4$). They define a polynomial $y^2+p_3(x,y)+p_4(x,y)$ whose homogenization is an element of $X_2$. It is clear now that a polynomial of this form is the product of a line and a cubic if and only if the coefficient of $x^3$ and $x^4$ are zero. Moreover, the condition that $(0,0)$ is the only intersection between the line and the cubic is equivalent to asking that the coefficient of $x^2y$ is zero.
\end{proof}

\begin{remark}
	 A straightforward computation gives use the fundamental class of $X_5^0$ and the strategy described at the beginning of \Cref{sec:3} gives us the description of the fundamental class of $\cA_5^0$ in the Chow ring of $\Mtilde_3$.
	We do not write down the explicit description because it is contained inside the ideal generated by the other relations. 
\end{remark}

\subsection*{Fundamental class of $A_n$-singularity}

We finally deal with the computation of the remaining fundamental classes. As usual, our strategy assures us that it is enough to compute the restriction of every fundamental class to every stratum. We do not give detail for every fundamental class. We instead describe the strategy to compute all of them in every stratum and leave the remaining computations to the reader.  

We start with the open stratum $\Mtilde_3 \setminus (\Htilde_3 \cup \Detilde_1)$, which is also the most difficult one. Luckily, we have already done all the work we need for the previous case. We adopt the same exact idea we used for the computation of $\cA_5^0$. 

\begin{remark}
	First of all, we can reduce the computation to the fundamental class of the locus $X_n$ in $(\PP^{14}\times \PP^{2})\setminus X$ parametrizing $(f,p)$ such that $p$ is an $A_h$-singularity for $h\geq n$. As above, $X$ is the closed locus paramatrizing $(f,p)$ such that $p$ is a singular point of $f$ but not an $A$-singularity.  Consider the morphism 
	$$\pi: \PP^{14}\times \PP^{2} \rightarrow \PP^{14}$$
	and consider the restriction 
	$$\pi\vert_{X_n}: X_n \longrightarrow\cA_n\setminus (\Htilde_3 \cup \Detilde_1);$$
	this is finite birational because generically a curve in $\cA_n$ has only one singular point. Therefore it is enough to compute the $H$-equivariant class of $X_n$ in $(\PP^{14}\times \PP^{2})\setminus X$ and then compute the pushforward $\pi_*(X_n)$. This is an exercise with Segre classes.
	We give the description of the relevant strata in the Chow ring of $\Mtilde_3$ in \Cref{rem:relations-Mbar}.
\end{remark}

\begin{proposition}
	In the situation above, we have that 
	$$[X_n]=C_ni_{2,*}(1) \in \ch_{\GL_3}((\PP^{14}\times \PP^2)\setminus X)$$
	where
	$$i_{2,*}(1)=2(h+k-c_1)(h+4k)((h+3k)^2-(c_1+k)(h+2k)+c_2)$$ 
	while $C_2=1$ and $C_n=c_3c_4\dots c_{n}$ for $n\geq 3$ where 
	$$ c_m:= -mc_1+\frac{2m-1}{2}h+(4-m)k  $$ for every $3\leq m \leq 7$.
\end{proposition}

\begin{proof}
  \Cref{prop:i-2} and \Cref{lem:X-2} imply that it is enough to compute the fundamental class of $X_n$ in $X_2$. It is important to remind that the coordinates of $X_2$ are the coefficients of the polynomial $p_3(x,y)+p_4(x,y)$ where $p_3$ and $p_4$ are homogeneous polynomials in $x,y$ of degree respectively $3$ and $4$. Moreover, if we see it as an element of $[\PP^{14}\times \PP^{2}/\GL_3]$, it is represented by the pair $(y^2z^2+p_3(x,y)z+p_4(x,y), [0:0:1])$.  Therefore, we need to find a relation between the coefficients of $p_3$ and $p_4$ such that the point $(0,0):=[0:0:1]$ is an $A_h$-singularity for $h \geq n$.
  
  To do so, we apply Weierstrass preparation theorem. Specifically, we use Algorithm 5.2 in \cite{Ell}, which allows us to write the polynomial $y^2+p_3(x,y) + p_4(x,y)$ in the form $y^2+p(x)y+q(x)$ up to an invertible element in $k[[x,y]]$. The square completion procedure implies that, up to an isomorphism of $k[[x,y]]$, we can reduce to the form $y^2+[q(x)-p(x)^2/4]$. Although $q(x)$ and $p(x)$ are power series, we just need to understand the coefficients of $h(x):=q(x)-p(x)^2/4$ up to degree $8$. Clearly the coefficient of $1$, $x$ and $x^2$ are already zero by construction. In general for $n\geq 3$, if $c_n$ is the coefficient of $x^{n}$ inside $h(x)$, we have that $X_n$ is the complete intersection inside $X_2$ of the hypersurfaces $c_i=0$ for $3\leq i \leq n$. We can use now the description of $X_2$ as a quotient stack (see \Cref{lem:X-2}) to compute the fundamental classes. 
\end{proof}

\begin{remark}
	Notice that for $\cA_5$, we also have the contribution of the closed substack $\cA_5^0$ that we need to remove to get the fundamental class of the non-separating locus. To same is not true for $\cA_3$, because $\cA_3^1$ is contained in the hyperelliptic locus.
\end{remark}

It remains to compute the fundamental class of $\cA_n$ restricted to the other strata. The easiest case is $\Detilde_{1,1,1}$, because clearly there are no $\cA_n$-singularities for $n\geq 3$. Therefore $\cA_n\vert_{\Detilde_{1,1,1}}=0$ for every $n\geq 3$. Regarding $\cA_2$, it enough to compute its pullback through the $6:1$-cover described in Lemma 6.1 of \cite{Per3}. We get the following result.

\begin{proposition}
	The restriction of $\cA_n$ to $\Detilde_{1,1,1}$ is of the form
	$$ 24(c_1^2-2c_2) \in \ZZ[1/6,c_1,c_2,c_3]\simeq  \ch(\Detilde_{1,1,1})$$
	for $n=2$ while it is trivial for $n\geq 3$.  
\end{proposition}

As far as $\Detilde_{1,1} \setminus \Detilde_{1,1,1}$ is concerned, we have that $\cA_n \cap \Detilde_{1,1} = \emptyset$ for $n\geq 4$. Moreover, $\cA_3 \cap \Detilde_{1,1}=\emptyset$ because every tacnode is separating in the stratum $\Detilde_{1,1}\setminus \Detilde_{1,1,1}$. Therefore again, we only need to do the computation for $\cA_2$, which is straightforward.

\begin{proposition}
	The restriction of $\cA_n$ to $\Detilde_{1,1}\setminus \Detilde_{1,1,1}$ is of the form
	$$ 24(t^2+c_1^2-2c_2) \in \ZZ[1/6,c_1,c_2,t]\simeq  \ch(\Detilde_{1,1}\setminus \Detilde_{1,1,1})$$
	for $n=2$ while it is trivial for $n\geq 3$.  
\end{proposition}

We now concentrate on the stratum $\Detilde_{1}\setminus \Detilde_{1,1}$. Recall that we have the isomorphism 
$$\Detilde_{1} \setminus \Detilde_{1,1}\simeq (\Mtilde_{2,1}\setminus \ThTilde_1) \times \Mtilde_{1,1}$$ 
as in Section 4 of \cite{Per3} and we denote by $t_0,t_1$ the generators of the Chow ring of $\Mtilde_{2,1}$ and by $t$ the generator of the Chow ring of $\Mtilde_{1,1}$.

\begin{proposition}
	We have the following description of the $A_n$-strata in the Chow ring of $\Detilde_{1}\setminus \Detilde_{1,1}$:
	\begin{itemize}
		\item[$\bullet$] the fundamental classes of $\cA_5$, $\cA_6$ and $\cA_7$ are trivial,
		\item[$\bullet$] the fundamental class of $\cA_4$ is equal to $40(t_0+t_1)^2t_0t_1$,
		\item[$\bullet$] the fundamental class of $\cA_3$ is equal to 
		$-24(t_0+t_1)^3 + 48(t_0+t_1)t_0t_1$,
		\item[$\bullet$] the fundamental class of $\cA_2$ is equal to 
		$24(t_0+t_1)^2 - 48t_0t_1+24t^2$.
	\end{itemize} 
\end{proposition}

\begin{proof}
	If $n=6,7$ the intersection of $\cA_n$ with $\Detilde_{1}\setminus \Detilde_{1,1}$ is empty. Notice that if $n=5$, it is also trivial as we are interested in non-separating singularities. It remains to compute the case for $n=2,3,4$. It is clear that if $n=3,4$, the factor $\Mtilde_{1,1}$ of the product does not give a contribution, therefore it is enough to describe the fundamental class of the locus of $A_n$-singularities in $\Mtilde_{2,1}\setminus \ThTilde_1$ for $n=3,4$. We do it exactly as in the proof of \Cref{lem:a-5-0}. The computation for $n=2$ again is straightforward. 
\end{proof}

Last part of the computations is the restriction to the hyperelliptic locus.

\begin{remark}\label{rem:sep-tac}
	We recall the stratification
	$$ \cA_3^1\setminus \Detilde_1  \subset \Xi_1\setminus \Detilde_1 \subset \Htilde_3\setminus \Detilde_1$$ 
	defined in the following way: $\Xi_1$ parametrizes triplets  $(Z,L,f)$ in $\Htilde_3\setminus \Detilde_1$ such that $Z$ is genus $0$ curve with one node whereas $\cA_3^1$ parametrizes triplets $(Z,L,f)$ in $\Xi_1$ such that $f$ vanishes at the node. Using the results in Section 2 of \cite{Per3}, we get that 
	\begin{itemize}
		\item $\ch(\Htilde_3\setminus (\Detilde_1 \cup \Xi_1)) \simeq \ZZ[1/6,s,c_2]/(f_9)$,
		\item $\ch(\Xi_1\setminus (\Detilde_1 \cup \cA_3^1))\simeq \ZZ[1/6,c_1,c_2]$,
		\item $\ch(\cA_3^1 \setminus \Detilde_1)\simeq \ZZ[1/6,s]$
	\end{itemize}
	where $f_9$ is the restriction of the relation $c_9$ to the open stratum (see Remark 2.15 of \cite{Per3}). Furthermore, the normal bundle of the closed immersion
	$$ \Xi_1\setminus (\Detilde_1\cup \cA_3^1) \into \Htilde_3 \setminus (\Detilde_1 \cup \cA_3^1)$$ 
	is equal to $-c_1$ whereas the normal bundle of the closed immersion 
	$$\cA_3^1 \setminus \Detilde_1 \into \Xi_1\setminus \Detilde_1$$
	is equal to $-2s$. 
\end{remark}

\Cref{lem:gluing} implies that we can compute the restriction of $\cA_n$ to the hyperelliptic locus using the stratification $$ \cA_3^1\setminus \Detilde_1  \subset \Xi_1\setminus \Detilde_1 \subset \Htilde_3\setminus \Detilde_1,$$
i.e. it is enough to compute the restriction of $\cA_n$ to $\cA_3^1\setminus \Detilde_1$, $\Xi_1\setminus (\Detilde_1\cup \cA_3^1)$ and $\Htilde_3\setminus (\Detilde_1 \cup \Xi_1)$.

\begin{lemma}
	The restriction of $\cA_n$ to $\cA_3^1\setminus \Detilde_1$ is empty if $n\geq 3$ whereas we have the equality $$[\cA_2]\vert_{ \cA_3^1\setminus \Detilde_1}=72s^2$$
	in the Chow ring of $\cA_3^1\setminus \Detilde_1$.

	Moreover, the restriction of $\cA_n$ to $\Xi_1\setminus (\Detilde_1\cup\cA_3^1)$ is empty if $n\geq 4$ whereas we have the two equalities
	\begin{itemize}
	\item$ [\cA_2]\vert_{\Xi_1\setminus (\Detilde_1 \cup\cA_3^1)} =24c_1^2-48c_2$ 
	\item$ [\cA_3]\vert_{\Xi_1\setminus (\Detilde_1\cup \cA_3^1)} = 24c_1^3-72c_1c_2$
	\end{itemize}
	in the Chow ring of $\Xi_1\setminus (\Detilde_1\cup \cA_3^1)$.
\end{lemma}

\begin{proof}
	This is an easy exercise and we leave it to the reader.
\end{proof}

\subsection*{Restriction to $\Htilde_3 \setminus (\Xi_2 \cup \Detilde_1)$}

It remains to compute the restriction of $\cA_n$ to the open stratum $\Htilde_3 \setminus (\Xi_2 \cup \Detilde_1)$.

\begin{remark}\label{rem:ar-vis-2}
	 We know that $\Htilde_3 \setminus (\Detilde_1\cup \Xi_1)$ is isomorphic to the quotient stack $[\AA(8)\setminus 0/(\GL_2/\mu_4)]$, where $\AA(8)$ is the space of homogeneous forms in two variables of degree $8$. Moreover, we have the isomorphism $\GL_2/\mu_4 \simeq \PGL_2\times \gm$. See Lemma 2.18 of \cite{Per3} or \cite{ArVis} for a more detailed discussion.
	 
	 Therefore we have that 
	 $$\ch(\Htilde_3\setminus (\Detilde_1\cup \Xi_1))\simeq \ZZ[1/6,s,c_2]$$ where $s$ is the first Chern class of the standard representation of $\gm$ while $c_2$ is the generator of the Chow ring of $\cB\PGL_2$. Notice that we have a morphism 
	 $$ \GL_2 \longrightarrow \PGL_2 \times \gm$$
	 defined by the association $A\mapsto ([A],\det{A}^2)$ which coincides with the natural quotient morphism 
	 $$ q:\GL_2 \longrightarrow \GL_2/\mu_4.$$
	 We have that $q^*(s)=2d_1$ and $q^*(c_2)=d_2$, where $d_1$ and $d_2$ are the first and second Chern class of the standard representation of $\GL_2$. 	
\end{remark}
 
Using the description of $\Htilde_3\setminus (\Detilde_1\cup \Xi_1)$ as a quotient stack highlighted in the previous remark, the restriction of $\cA_n$ to $\Htilde_3\setminus  (\Detilde_1\cup \Xi_1)$ is the locus $H_n$ which parametrizes forms $f \in \AA(8)$ such that $f$ has a root of multiplicity at least $n+1$. Thanks to \Cref{rem:ar-vis-2}, it is enough to compute the fundamental class of $H_n$ after pulling it back through $q$, i.e. compute the $\GL_2$-equivariant fundamental class of $H_n$. Because $\GL_2$ is a special group, we can reduce to do the computation of the $T$-equivariant fundamental class of $H_n$, where $T$ is the torus of diagonal matrices in $\GL_2$. Therefore, we can use the formula in \Cref{rem:gener} to get the explicit description of the $T$-equivariant class of $H_n$.

\section{The Chow ring of $\Mbar_3$ and the comparison with Faber's result}\label{sec:4}

We are finally ready to present our description of the Chow ring of $\Mbar_3$. 

\begin{theorem}\label{theo:chow-ring-m3bar}
	 Let $\kappa$ be the base field of characteristic different from $2,3,5,7$. The Chow ring of $\Mbar_3$ with $\ZZ[1/6]$-coefficients is the quotient of the graded polynomial algebra 
	 $$\ZZ[1/6,\lambda_1,\lambda_2,\lambda_3,\delta_{1},\delta_{1,1},\delta_{1,1,1},H]$$
	 where 
	 \begin{itemize}
	 	\item[] $\lambda_1,\delta_1,H$ have degree $1$, \item[]$\lambda_2,\delta_{1,1}$ have degree $2$, \item[]$\lambda_3,\delta_{1,1,1}$ have degree $3$.
 	\end{itemize}
 	The quotient ideal is generated by 15 homogeneous relations: 
	 \begin{itemize}
	 	\item[] $[\cA_2]$, which is in codimension $2$;
	 	\item[] $[\cA_3], [\cA_3^{1}], \delta_1^c, k_1(1),k_{1,1}(2)$, which are in codimension $3$,
	 	\item[] $[\cA_4], \delta_{1,1}^c, k_{1,1}(1), k_{1,1,1}(1), k_{1,1,1}(4), m(1),k_h,k_1(2)$, which are in codimension $4$,
	 	\item[] $ k_{1,1}(3)$, which is in codimension $5$.
	 \end{itemize}
\end{theorem}

\begin{remark}\label{rem:relations-Mbar}
	We write the relations explicitly. 
	\begin{itemize}

		\item[]
		\begin{equation*}
			\begin{split}
			[\cA_2]=24(\lambda_1^2-2\lambda_2)
			\end{split}
		\end{equation*}
		\item[]
		\begin{equation*}
			\begin{split}
				[\cA_3]&= 36\lambda_1^3 + 10\lambda_1^2H + 21\lambda_1^2\delta_1 - 92\lambda_1\lambda_2 - 4\lambda_1H^2 + 18\lambda_1H\delta_1 + \\ & +72\lambda_1\delta_1^2 +
				+ 88\lambda_1\delta_{1,1} - 20\lambda_2H + 56\lambda_3 - 2H^3 + 9H^2\delta_1 + 54H\delta_1^2 + \\ & + 87\delta_1^3 -
				4\delta_1\delta_{1,1}+ 56\delta_{1,1,1}
			\end{split}
		\end{equation*}

		\item[]
		\begin{equation*}
			\begin{split}
				[\cA_3^1]=\frac{H}{2}(\lambda_1+3H+\delta_1)(\lambda_1+H+\delta_1)
			\end{split}
		\end{equation*}
		 
		\item[]
		\begin{equation*}
			\begin{split}
				\delta_1^c&=6(\lambda_1^2\delta_1 + 2\lambda_1H\delta_1 + 6\lambda_1\delta_1^2 + 4\lambda_1\delta_{1,1} + H^2\delta_1 + 6H\delta_1^2 -\\ & - 4H\delta_{1,1}
				+ 9\delta_1^3 - 8\delta_1\delta_{1,1} + 12\delta_{1,1,1})
			\end{split}
		\end{equation*}
		\item[]
		\begin{equation*}
			\begin{split}
			k_1(1)=\delta_1\Big(\frac{1}{4}\lambda_1^2 + \frac{1}{2}\lambda_1H + 2\lambda_1\delta_1 + \lambda_2 + \frac{1}{2}H^2 + H\delta_1^2 + \frac{7}{4}\delta_1^2 -\delta_{1,1}\Big)
			\end{split}
		\end{equation*} 
		\item[]
		\begin{equation*}
			\begin{split}
				k_{1,1}(2)=\delta_{1,1}(3\lambda_1 + H + 3\delta_1)
			\end{split}
		\end{equation*} 
		\item[] 
		\begin{equation*}
			\begin{split}
				[\cA_4]&=36\lambda_1^4 + \frac{1048}{27}\lambda_1^3H + 66\lambda_1^3\delta_1 - 92\lambda_1^2\lambda_2 - \frac{146}{81}\lambda_1^2H^2 +
				\frac{517}{9}\lambda_1^2H\delta_1 + \\ & + 207\lambda_1^2\delta_1^2 - 176\lambda_1^2\delta_{1,1} - 84\lambda_1\lambda_2H + 56\lambda_1\lambda_3 +
				\frac{16}{81}\lambda_1H^3 +\\ & + \frac{3272}{81}\lambda_1H^2\delta_1 +  \frac{1282}{9}\lambda_1H\delta_1^2 + 222\lambda_1\delta_1^3 -
				340\lambda_1\delta_1\delta_{1,1} + \\ & + 56\lambda_1\delta_{1,1,1} + 8\lambda_2H^2 + \frac{130}{27}H^4  + \frac{2041}{81}H^3\delta_1 +
				\frac{4957}{81}H^2\delta_1^2 + \\ & + \frac{2101}{27}H\delta_1^3 + 45\delta_1^4 - 72\delta_1^2\delta_{1,1}
			\end{split}
		\end{equation*}
		\item[]
		\begin{equation*}
			\begin{split}
				\delta_{1,1}^c=24\big(\delta_{1,1}(\lambda_1+\delta_1)^2+\delta_{1,1,1}\big)
			\end{split}
		\end{equation*} 
		\item[]
		\begin{equation*}
			\begin{split}
				k_{1,1}(1)=\delta_{1,1}(\delta_{1,1}-\lambda_2-(\lambda_1+\delta_1)^2)
			\end{split}
		\end{equation*} 
		\item[] 
		\begin{equation*}
			\begin{split}
				k_{1,1,1}(1)=\delta_{1,1,1}(\lambda_1 + \delta_1)
			\end{split}
		\end{equation*}
		\item[]
		\begin{equation*}
			\begin{split}
				k_{1,1,1}(4)=H\delta_{1,1,1}
			\end{split}
		\end{equation*}
		\item[]
		\begin{equation*}
			\begin{split}
				m(1)&=12\lambda_1^4 - \frac{7}{3}\lambda_1^3H + 27\lambda_1^3\delta_1 - 44\lambda_1^2\lambda_2 - \frac{706}{9}\lambda_1^2H^2 - \frac{65}{2}\lambda_1^2H\delta_1 + \\ &
				+ 84\lambda_1^2\delta_1^2 - 32\lambda_1^2\delta_{1,1} - 38\lambda_1\lambda_2H + 92\lambda_1\lambda_3 - \frac{715}{9}\lambda_1H^3 - \\ & -
				\frac{1340}{9}\lambda_1H^2\delta_1 - 25\lambda_1H\delta_1^2  + 69\lambda_1\delta_1^3 - 130\lambda_1\delta_1\delta_{1,1} + 92\lambda_1\delta_{1,1,1} + \\ & +
				6\lambda_2H^2 - \frac{46}{3}H^4 - \frac{1205}{18}H^3\delta_1  - \frac{562}{9}H^2\delta_1^2 - \frac{101}{6}H\delta_1^3 -
				54\delta_1^2\delta_{1,1}
			\end{split}
		\end{equation*} 
		\item[] 
		\begin{equation*}
			\begin{split}
				k_h&= \frac{1}{8}\lambda_1^3H + \frac{1}{8}\lambda_1^2H^2 + \frac{1}{4}\lambda_1^2H\delta_1 - \frac{1}{2}\lambda_1\lambda_2H - \frac{1}{8}\lambda_1H^3 +
				\frac{7}{8}\lambda_1H\delta_1^2 + \\ & + \frac{3}{2}\lambda_1\delta_1\delta_2   \frac{1}{2}\lambda_2H^2 + \lambda_3H  - \frac{1}{8}H^4 - \frac{1}{4}H^3\delta_1 +
				\frac{1}{8}H^2\delta_1^2 + \frac{3}{4}H\delta_1^3 + \frac{3}{2}\delta_1^2\delta_2
			\end{split}
		\end{equation*}
			\item[]
		\begin{equation*}
			\begin{split}
				k_1(2)&=\frac{1}{4}\lambda_1^3\delta_1 + \frac{1}{2}\lambda_1^2H\delta_1 + \frac{5}{4}\lambda_1^2\delta_1^2 + \frac{1}{4}\lambda_1H^2\delta_1 + \frac{3}{2}\lambda_1H\delta_1^2 +
				\frac{7}{4}\lambda_1\delta_1^3 +\\ & + \lambda_1\delta_1\delta_2 - \lambda_1\delta_3 + \lambda_3\delta_1+ \frac{1}{4}H^2\delta_1^2 + H\delta_1^3 + \frac{3}{4}\delta_1^4 +
				\delta_1^2\delta_2
			\end{split}
		\end{equation*}
		\item[]
		\begin{equation*}
			\begin{split}
				k_{1,1}(3)&=2\lambda_1^3\delta_2 + 5\lambda_1^2\delta_1\delta_2 + \lambda_1\lambda_2\delta_2 + 4\lambda_1\delta_1^2\delta_2 + \lambda_2\delta_1\delta_2 + \lambda_2\delta_3 + \lambda_3\delta_2 +\\ & +
				\delta_1^3\delta_2
			\end{split}
		\end{equation*}
	\end{itemize}
\end{remark}

\begin{remark}
	Notice that the relation $[\cA_2]$ and $\delta_1^c$ gives us that $\lambda_2$ and $\delta_{1,1,1}$ can be obtained using the other generators. 
\end{remark}

Lastly, we compare our result with the one of Faber, namely Theorem 5.2 of \cite{Fab}. Recall that he described the Chow ring of $\Mbar_3$ with rational coefficients as the graded $\QQ$-algebra defined as a quotient of the graded polynomial algebra generated by $\lambda_1$, $\delta_1$, $\delta_0$ and $\kappa_2$. We refer to \cite{Mum} for a geometric description of these cycles. The quotient ideal is generated by $3$ relations in codimension $3$ and six relations in codimension $4$. 

First of all, if we invert $7$, we have that the relation $[\cA_3]$ implies that also $\lambda_3$ is not necessary as a generator. Therefore if we tensor with $\QQ$, our description can be simplified and we end up having exactly $4$ generators, namely $\lambda_1$, $\delta_1$, $\delta_{1,1}$ and $H$ and $9$ relations. Notice that the identity 
$$ [H]=9\lambda_1 - 3\delta_1 - \delta_0$$
allows us to easily pass from the generator $H$ to the generator $\delta_0$ that was used in \cite{Fab}. Finally, Table 2 in \cite{Fab} gives us the identity
$$ \delta_{1,1}=-5\lambda_1^2+\frac{\lambda_1\delta_0}{2} +\lambda_1\delta_1+\frac{\delta_1^2}{2}+\frac{\kappa_2}{2}$$ 
which explains how to pass from the generator $\delta_{1,1}$ to the generator $\kappa_2$ used in \cite{Fab}.

Thus we can construct two morphisms of $\QQ$-algebras:
$$ \phi: \QQ[\lambda_1,H,\delta_1,\delta_{1,1}]\longrightarrow \QQ[\lambda_1, \delta_0,\delta_1,\kappa_2] $$
and 
$$ \varphi: \QQ[\lambda_1,\delta_0,\delta_1,\kappa_2] \longrightarrow \QQ[\lambda_1,H,\delta_1,\delta_{1,1}]$$
which are  one the inverse of the other. A computation shows that $\phi$ sends our ideal of relations to the one in \cite{Fab} and $\varphi$ sends the ideal of relations in \cite{Fab} to the one we constructed. 

\appendix
\section{Discriminant relations}	

In this appendix, we generalize Proposition 4.2 of \cite{EdFul}. We do not need this result in its full generality in our work, only the formulas in \Cref{rem:gener}.

First of all, we set some notations. Everything is considered over a base field $\kappa$. Let $T$ be the $2$-dimensional split torus $\gm^2$ which embeds in $\GL_2$ as the diagonal matrices and let $E$ be the standard representation of $\GL_2$. Let $n$ be  positive integer. We denote by $\AA(n)$ the $n$-th symmetric power of the dual representation of $E$ and by $\PP^n$  the projective bundle $\PP(\AA(n))$. We denote by $\xi_n$ the hyperplane sections of $\PP^n$. Moreover, we denote by $h_i$ the element of $\ch_T(\PP^n)$ associated to the hyperplane defined by the equation $a_{i,n-i}=0$ for every $i=0,\dots,n$ where $a_{i,n-i}$ is the coordinate of $\PP^n$ associated to the coefficient of $x_0^ix_1^{n-i}$ and $x_0,x_1$ is a $T$-base for $E^{\vee}$. We have the identity
$$ h_i=\xi_n -(n-i)t_0- it_1$$ 
where $t_0,t_1$ are the generators of $\ch(\cB T)$ (acting respectively on $x_0$ and $x_1$). Let $\tau \in \ch(\cB T)$ be the element $t_0-t_1$, then the previous identity can be written as 
$$ h_i=h_0+i\tau.$$
Notice that we can reduce to the $T$-equivariant setting exactly as the authors do in \cite{EdFul}, because $\GL_2$ is a special group and therefore the morphism 
$$ \ch(\cB \GL_2) \longrightarrow \ch(\cB T)$$ 
is injective
 
Let $N$ and $k$ be two positive integers such that $k\leq N$. Inside $\PP^N$, we can define a closed subscheme $\Delta_k$ parametrizing (classes of) homogeneous forms in two variables $x_0,x_1$ which have a root of multiplicity at least $k$. 

We want to study the image of the pushforward of the closed immersion 
$$ \Delta_k \into  \PP^N;$$
we have the description of the Chow ring of $\PP^N$ as the quotient 
$$ \ch_{\GL_2}(\PP^N) \simeq  \ZZ[c_1,c_2,\xi_N]/(p_N(\xi_N))$$
where $p_N(\xi_N)$ is a monic polynomial in $h$ of degree $N+1$ with coefficients in $\ch(\cB\GL_2)\simeq \ZZ[c_1,c_2]$. The coefficient of $\xi_N^i$ is the $(N-i)$-th Chern class of the $\GL_2$-representation $\AA(N)$ for $i=0,\dots,N$. 

Exactly as it was done in \cite{Vis3} and generalized in \cite{EdFul2}, we introduce the multiplication morphism for every positive integer $r$ such that $r\leq N/m$
$$ \pi_r: \PP^r \times \PP^{N-kr} \longrightarrow \PP^N$$ 
defined by the association $(f,g)\mapsto f^kg$.
The $\GL_2$-action on the left hand side is again induced by the symmetric powers of the dual of $E$. Notice that we are not assuming that $N$ is a multiple of $k$. We have an analogue of Proposition 3.3 of \cite{Vis3} or Proposition 4.1 of \cite{EdFul2}.

\begin{proposition}
	Suppose that the characteristic of $\kappa$ is greater than $N$, then the disjoint union of the morphisms $\pi_r$ for $1\leq r\leq N/k$ is a Chow envelope for $\Delta_k\into \PP^N$. 
\end{proposition}

Therefore it is enough to study the image of the pushforward of $\pi_r$ for $r\leq N/k$.  We have that $\pi^*(\xi_N)=k\xi_r+\xi_{N-kr}$, therefore for a fixed $r$ we have that the image of $\pi_{r,*}$ is generated as an ideal by $\pi_{r,*}(\xi_r^m)$ for $0\leq m\leq r$.

\begin{remark}
	Fix $r\leq N/k$. We have that 
	$$ \pi_{r,*}(\xi_r^m) \in \big(\pi_{r,*}(1), \pi_{r,*}(h_0), \dots, \pi_{r,*}(h_0\dots h_{m-1})\big)$$ 
	in $\ch_T(\PP^r)$ for $m\leq r$. In fact, we have 
	$$ h_0\dots h_{m-1}= \prod_{i=0}^{m-1} (\xi_r-(r-i)t_0-it_1) = \xi_r^m + \sum_{i=0}^{m-1}\alpha_i\xi_r^i $$
	with $\alpha_i \in \ch(\cB T)$. Therefore we can prove it by induction on $m$. 
\end{remark}

Therefore, it is enough to describe the ideal generated by $\pi_{r,*}(h_0\dots h_m)$ for $1\leq r\leq N/k$ and $-1\leq m \leq r-1$. We define the element associated to $m=-1$ as $\pi_{r,*}(1)$. 

Our goal is to prove that the ideal is in fact generated by $\pi_{1,*}(1)$ and $\pi_{1}(h_0)$. To do so, we have to introduce some morphisms first. 

Let $n$ be an integer and $\rho_n:(\PP^1)^{\times n} \rightarrow \PP^n$ by the $n$-fold product morphism, which is an $S_n$-torsor, where $S_n$ is the $n$-th symmetric group. Furthermore, we denote by $\Delta^n:\PP^1 \into (\PP^1)^{\times n}$ the small diagonal in the $n$-fold product, i.e. the morphism defined by the association $f \mapsto (f,f,\dots,f)$. We denote by $h_i$  the fundamental class of $[\infty]:=[0:1]$ in the Chow ring of the $i$-th element of the product $(\PP^1)^n$ (and by pullback in the Chow ring of the product).
\begin{remark}
	Notice that we are using the same notation for two different elements: if we are in the projective space $\PP^n$, $h_i$ is the hyperplane defined by the vanishing of the $i+1$-th coordinate of $\PP^n$. On the contrary, if we are in the Chow ring of the product $(\PP^1)^n$, it represents the subvariety defined as the pullback through the $i$-th projection of the closed immersion $\infty \into \PP^1$. Notice that $\rho_{n,*}(h_0\dots h_s)$ is equal  to $s! (n-s)! h_0\dots h_s$ for every $s\leq n$.
\end{remark}
We have a commutative diagram of finite morphisms 
$$
\begin{tikzcd}
	(\PP^1)^r \times (\PP^1)^{N-kr} \arrow[r, "\alpha_r^k"] \arrow[d, "\rho_r \times \rho_{N-kr}"'] & (\PP^1)^N \arrow[d, "\rho_N"] \\
	\PP^r \times \PP^{N-kr} \arrow[r, "\pi_r"]                                                    & \PP^N                        
\end{tikzcd}
$$  
where $\alpha_r^k= (\Delta^k)^{\times r} \times \id_{(\PP^1)^{N-kr}}$.  We can use this diagram to have a concrete description of $\pi_{r,*}(h_0\dots h_m)$. In order to do so, we first need the following lemma to describe the fundamental  class of the image of $\alpha_r^k$. 

\begin{lemma}\label{lem:k-diag}
We have the following identity 
	$$ [\Delta^k]= \sum_{j=0}^{k-1} \tau^{k-1-j}\sigma_j^k(h_1,\dots,h_k)$$ 
	in the Chow ring of $(\PP^1)^{\times k}$ for every $k \geq 2$, where $\sigma_j^k(-)$ is the elementary symmetric function with $k$ variables of degree $j$. 
\end{lemma}

\begin{proof}
The diagonal $\Delta^k$ is equal to the complete intersection of the hypersurfaces of $(\PP^1)^k$
of equations $x_{0,i}x_{1,i+1}-x_{0,i+1}x_{1,i}$ for $1\leq i \leq k-1$ (we are denoting  by $x_{0,i},x_{i,1}$ the two coordinates of the $i$-th factor of the product). Therefore we have 
$$ \Delta^k=\prod_{i=1}^{k-1}(h_i+h_{i+1}+\tau). $$
Notice that in the Chow ring of $(\PP^1)^k$ we have $k$ relations of degree $2$ which can be written as $h_i^2+\tau h_i=0$ for every $i=1,\dots,k$. 

The case $k=2$ was already proven in Lemma 3.8 of \cite{Vis3}. We proceed by induction on $k$. We have
$$\Delta^{k+1}=\prod_{i=1}^{k}(h_i+h_{i+1}+\tau)=(h_{k+1}+h_k+\tau) \Delta^k$$ 
and thus by induction 
$$\Delta^{k+1}=\sum_{i=0}^{k-1}h_{k+1} \tau^{k-1-i}\sigma_i^k + \sum_{i=0}^{k-1}  \tau^{k-i}\sigma_i^k  +\sum_{i=0}^{k-1} h_k \tau^{k-1-i}\sigma_i^k.$$ 
Recall that we have the relations $$\sigma_j^k(x_1,\dots,x_k)=x_k\sigma_{j-1}^{k-1}(x_1,\dots,x_{k-1}) + \sigma_j^{k-1}(x_1,\dots,x_{k-1})$$ between elementary symmetric functions (with $\sigma_j^k=0$ for $j>k$), therefore we have
$$ \sum_{i=0}^{k-1} \tau^{k-1-i}  h_k\sigma_i^k = \sum_{i=0}^{k-1} \tau^{k-1-i}(h_k^2\sigma_{i-1}^{k-1}+h_k\sigma_{i}^{k-1})=\sum_{i=0}^{k-1} \tau^{k-1-i}h_k(-\tau \sigma_{i-1}^{k-1} +\sigma_i^{k-1})$$ 
where we used the relation $h_k^2+\tau h_k=0$ in the last equalities. Therefore we get
\begin{equation*}
	\begin{split}
		\Delta^{k+1}&=\sum_{i=0}^{k-1} \tau^{k-1-i}(h_{k+1}\sigma_i^k+h_k\sigma_i^{k-1}) + \sum_{i=0}^{k-1} \tau^{k-i}(\sigma_i^k-h_k\sigma_{i-1}^{k-1})= \\ & =\sum_{i=0}^{k-1}\tau^{k-1-i}(h_{k+1}\sigma_i^k+h_k\sigma_i^{k-1}) + \sum_{i=0}^{k-1} \tau^{k-i}\sigma_i^{k-1}
	\end{split} 
\end{equation*}
Shifting the index of the last sum, it is easy to get the following identity
$$ \Delta^{k+1}= h_{k+1}\sigma_{k-1}^k+h_k\sigma_{k-1}^{k-1} + \tau^k +\sum_{i=0}^{k-2}\tau^{k-1-i}(h_{k+1}\sigma_i^k+h_k\sigma_i^{k-1}+\sigma_{i+1}^{k-1})$$
and the statement follows from shifting the last sum again and from using the relations between the symmetric functions (notice that $h_k\sigma_{k-1}^{k-1}=\sigma_k^k$).
\end{proof}

\begin{remark}
	We define by $\theta_{m,r}$ the $T$-equivariant closed subvariety of $(\PP^1)^r \times (\PP^1)^{N-kr}$ of the form 
	$$\theta_{m,r}:=\infty^{m+1} \times (\PP^1)^{r-(m+1)} \times (\PP^1)^{N-kr}$$
	induced by the $T$-equivariant closed immersion $\infty \into \PP^1$.
	We  have that $$(\rho_r \times \rho_{N-kr})_*(\theta_{m.r})=(r-(m+1))! (N-kr)! h_0\dots h_m$$ for every $m\leq r-1$. 
\end{remark}

From now on, we set $d:=r-(m+1)\geq 0$. Thanks to the remark and the commutativity of the diagram we constructed, we can computate the pushforward $\rho_{N,*}\alpha_{r,*}^k(\theta_{m,r})$ and then divide it by $d!(N-kr)!$ to get $\pi_{r,*}(h_0\dots h_m)$.

We denote by $\alpha_l^{(k,d)}$ the integer
$$ \alpha_{(k,d)}:=\sum_{j_1+\dots+j_d=l}^{0\leq j_s \leq k-1}\binom{k}{j1}\dots \binom{k}{j_d}$$
and by $\beta_l^{(k,m,r)}$ the integer (because $l\leq d(k-1)$)
$$ \beta_l^{(k,m,r)}:=\frac{(N-(m+1)k-l)!}{(N-kr)! d!}.$$

\begin{lemma}\label{lem:pi-r-1}
	We get the following equality 
	$$ \pi_{r,*}(h_0\dots h_m) = \sum_{l=0}^{d(k-1)} \alpha_l^{(k,d)}\beta_{l}^{(k,m,r)} \tau^{d(k-1)-l} h_0\dots h_{(m+1)k+l-1}$$
	in the $T$-equivariant Chow ring of $\PP^N$.
\end{lemma} 

\begin{proof}
	Thanks to \Cref{lem:k-diag}, we have that 
	\begin{equation*}
		\begin{split}
			&\alpha_{r,*}^k(\theta_{m,r})=[\infty^{k(m+1)} \times (\Delta^k)^d \times (\PP^1)^{N-kr}]= \\ & =h_1\dots h_{(m+1)k} \prod_{i=1}^d \Big(\sum_{j=0}^{k-1} \tau^{k-1-j}\sigma_j^{k}(h_{(m+j)k+1},h_{(m+j)k+2},\dots,h_{(m+j+1)k})\Big);
		\end{split}
	\end{equation*}
	we need to take the image through $\rho_{N,*}$ of this element. However, we have that $\rho_{N,*}(h_{i_1} \dots h_{i_s})=\rho_{N,*}(h_1\dots h_s)$ for every $s$-uple of $(i_1,\dots,i_s)$ of distinct indexes because $\rho_N$ is a $S_N$-torsor. Therefore a simple computation shows that $\rho_{N,*}\alpha_{r,*}^k(\theta_{m,r})$ has the following form:
	$$\sum_{l=0}^{d(k-1)}\Big(\sum_{j_1+\dots+j_d=l}^{0\leq j_s \leq k-1}\binom{k}{j_1}\dots \binom{k}{j_d} \Big)(N-(m+1)k-l)!  \tau^{d(k-1)-l} h_0 \dots   h_{(m+1)k+l-1}.$$
	The statement follows.
\end{proof}

\begin{remark}\label{rem:gener}
	Notice that the expression makes sense also per $m=-1$ and in fact we get a description of the $\pi_{r,*}(1)$. 
	
	Let us describe the case $r=1$. Clearly we only have $d=1$ (or $m=-1$) and $d=0$ (or $m=0$). If $d=0$, the formula gives us 
	$$\pi_{1,*}(h_0)=h_0\dots h_{k-1} \in \ch_T(\PP^N);$$
	if $d=1$ a simple computation shows 
	$$ \pi_{1,*}(1)=\sum_{l=0}^{k-1}(k-l)!\binom{k}{l}\binom{N-k}{N-l}\tau^{k-1-l}h_0\dots h_{l-1} \in \ch_T(\PP^N).$$
	These two formula gives us the $T$-equivariant class of these two elements. As matter of fact, $\pi_{1,*}(1)$ is also a $\GL_2$-equivariant class by definition. As far as $\pi_{1,*}(h_0)$ is concerned, this is clearly not $\GL_2$-equivariant. Nevertheless we can consider
	 $$\pi_ {1,*}(\xi_1)=\pi_{1,*}(h_0)+t_1\pi_{1,*}(1)$$
	which is a $\GL_2$ -equivariant class.
	
	We can describe $\pi_{1,*}(1)$ geometrically. In fact, it is the fundamental class of the locus describing forms $f$ such that $f$ has a root with multiplicity at least $k$. This locus is strictly related to locus of $A_k$-singularities in the moduli stack of cyclic covers of the projective line of degree $2$.
\end{remark}

We denote by $I$ the ideal generated by the two elements described in the previous remark. First of all, we prove that almost all the pushforwards we need to compute are in this ideal. 

\begin{proposition}
	We have that
	$$ \pi_{r,*}(h_0\dots h_m) \in I $$
	for every $1\leq r \leq N/k$ and $0\leq m \leq r-1$. 
\end{proposition}
\begin{proof}
	\Cref{lem:pi-r-1} implies that it is enough to prove that $(m+1)k+l-1\geq k-1$ for every $l=0,\dots, d(k-1)$ where $d=r-(m+1)$, because it implies that every factor of $\pi_{r,*}(h_0\dots h_m)$ is divisible by $h_0\dots h_{k-1}$. This follows from $m \geq 0$. 
\end{proof}

Therefore it only remains to prove  that $\pi_{r,*}(1)$ is in the ideal $I$ for $r\geq 2$. To do so, we need to prove some preliminary results.

\begin{proposition}\label{prop:square-power}
	We have the following equality 
	$$ h_0^2 \dots h_{n-1}^2 h_n \dots h_{m-1} = \sum_{s=0}^n (-1)^s s! \binom{n}{s}\binom{m}{s}\tau^s h_0\dots h_{m+n-s-1}$$ 
	in the $T$-equivariant Chow ring of $\PP^N$ for every $n\leq m$.
\end{proposition}

\begin{proof}
	Denote by $a_{n,m}$ the left term of the equality. Because we have the identity $h_i=h_j+(j-i)\tau$ in the $T$-equivariant Chow ring of $\PP^N$, we have the following formula
	$$ a_{n,m}=a_{n-1,m+1}-(m-n+1)\tau a_{n-1,m}$$ 
	which gives us that $a_{n,m}$ is uniquely determined from the elements $a_{0,j}$ for $j\leq N$. This implies that it is enough to prove that the formula in the statement verifies the recursive formula above. This follows from straightforward computation.
\end{proof}

Before going forward with our computation, we recall the following combinatorial fact.

\begin{lemma}\label{lem:comb}
	For every pair of non-negative integers $k,m\leq N$ we have that
	$$\sum_{l=0}^{k-1} (-1)^l \binom{m}{l}\binom{N-l}{k-1-l}=\binom{N-m}{k-1}.$$ 
\end{lemma}

We are going to use it to prove the following result.

\begin{proposition}\label{prop:square-h}
	For every non-negative integer $ t\leq k-1$, we have the following equality
	$$ h_0\dots h_{t-1} \pi_{1,*}(1) = \sum_{f=0}^{k-1} \frac{(N-f-t)!}{(N-k-t)!} \binom{k}{f}\tau^{k-1-f}h_0\dots h_{t+f-1} + I$$ 
	in the $T$-equivariant Chow ring of $\PP^N$. Again, for $t=0$, we end up with the formula for $\pi_{1,*}(1)$.
\end{proposition}
\begin{proof}
	The left hand side of the equation in the statement can be written as
	\begin{equation*}
		\begin{split}
			\sum_{l=0}^{t}(k-l)!\binom{k}{l}\binom{N-k}{N-l}\tau^{k-1-l}h_0^2\dots h_{l-1}^2 h_l\dots h_{t-1} + \\ + \sum_{l=t+1}^{k-1}(k-l)!\binom{k}{l}\binom{N-k}{N-l}\tau^{k-1-l}h_0^2\dots h_{t-1}^2 h_t \dots h_{l-1}.;
		\end{split}
	\end{equation*}
	see \Cref{rem:gener}.
	If we apply \Cref{prop:square-power} to the two sums, we get 
	$$ \sum_{l=0}^t \sum_{s=0}^l (-1)^s\frac{k! (N-l)! t!}{(k-l)! (N-k)! s! (l-s)! (t-s)!} \tau^{k-1-l+s}h_0\dots h_{l+t-s-1}$$ 
	and 
	$$ \sum_{l=t+1}^t \sum_{s=0}^t (-1)^s\frac{k! (N-l)! t!}{(k-l)! (N-k)! s! (l-s)! (t-s)!} \tau^{k-1-l+s}h_0\dots h_{l+t-s-1}; $$
	if we exchange the sums in each factor and put everything together we end up with
	$$ \sum_{s=0}^t \sum_{l=s}^{k-1} (-1)^s\frac{k! (N-l)! t!}{(k-l)! (N-k)! s! (l-s)! (t-s)!} \tau^{k-1-l+s}h_0\dots h_{l+t-s-1}. $$
	Shifting the inner sum and setting $f:=l-s$, we get
	$$ \sum_{s=0}^t \sum_{f=0}^{k-1-s} (-1)^s\frac{k! (N-s-f)! t!}{(k-s-f)! (N-k)! s! f! (t-s)!} \tau^{k-1-f}h_0\dots h_{l+f-1}.$$ 
	Notice that we can extend the inner sum up to $k-1$ as all the elements we are adding are in the ideal $I$. Therefore we exchange the sums again and get 
	$$ \sum_{f=0}^{k-1} (-1)^s\frac{k!}{f!} \tau^{k-1-f}h_0\dots h_{l+f-1}\Big( \sum_{s=0}^t (-1)^s \binom{N-s-f}{k-s-f}\binom{t}{s}\Big) + I $$
	and we can conclude using \Cref{lem:comb}.
\end{proof}

We state now the last technical lemma.

\begin{lemma}\label{lem:tec}
	If we define $\Gamma_t$ to be the element 
	$$ \frac{(N-t)!}{(N-2k+1)!} \tau^{2(k-1)-t}h_0\dots h_{t-1}$$
	in the $T$-equivariant Chow ring of $\PP^N$, we have that $\Gamma_t \in I$ for every $t\leq k-1$. 
\end{lemma}

\begin{proof}
	We proceed by induction on $m=k-1-t$. The case $m=0$ follows from the previous proposition. 
	
	Suppose that $\Gamma_s \in I$ for every $s\geq k-t$. If we consider the element in $I$ 
	$$ \frac{(N-k-t)!}{(N-2k+1)!}\tau^{k-1-t}h_0 \dots h_{t-1} \pi_{1,*}(1)$$ 
	we can apply the previous proposition again and get 
	$$ \sum_{f=0}^{k-1} \binom{k}{f}\frac{(N-f-t)!}{(N-2k+1)!} \tau^{2(k-1)-(f+t)}h_0\dots h_{t+f-1} \in I$$ 
	which is the same as 
	$$ \sum_{f=0}^{k-1} \binom{k}{f} \Gamma_{f+t} \in I$$. The statement follows by induction.
\end{proof}
\begin{remark}\label{rem:nec}
	It is important to notice that more is true, the same exact proof shows us that 
	$$ \Gamma_t \in \binom{2(k-1)-t}{k-1-t} \cdot I$$ 
	for every $t \leq k-1$. This will not be needed, except for the case $t=0$, where this implies that in particular $\Gamma_0 \in 2 \cdot I$.
\end{remark}
Before going to prove the final proposition, we recall the following combinatorial fact.

\begin{lemma}\label{lem:comb2}
	We have the following numerical equality 
	$$ \sum_{j_1+\dots +j_r=l}^{0\leq j_s \leq k-1} \binom{k}{j_1}\dots \binom{k}{j_r} = \binom{rk}{l}$$
	for every $l\leq k-1$. In particular in our situation we have $$\alpha_{l}^{k,r}=\binom{rk}{l}$$
	for $l\leq k-1$. 
\end{lemma}

Finally, we are ready to prove the last statement.

\begin{proposition}
	We have $\pi_{r,*}(1)$ is contained in the ideal $I$ for $r \geq 2$. 
\end{proposition}

\begin{proof}
	Notice that 
	\begin{equation*}
		\begin{split}
			\pi_{r,*}(1)&=\sum_{l=0}^{r(k-1)} \alpha_l^{(k,r)}\beta_l^{(k,-1,r)}\tau^{r(k-1)-l}h_0\dots h_{l-1}=\\ & =\sum_{l=0}^{k-1} \alpha_l^{(k,r)}\beta_l^{(k,-1,r)}\tau^{r(k-1)-l}h_0\dots h_{l-1} + I
		\end{split}
	\end{equation*}
	therefore we have to study the first $k-1$ elements of the sum. Using \Cref{lem:comb2}, we get the following chain of equalities modulo the ideal I;
	\begin{equation*}
		\begin{split}
			\pi_{r,*}(1)&= \sum_{l=0}^{k-1} \binom{rk}{l}\frac{(N-l)!}{(N-rk)!r!}\tau^{r(k-1)-l}h_0\dots h_l= \\ & =
			\sum_{l=0}^{k-1}\binom{rk}{l}\frac{(N-2k+1)!}{(N-rk)!r!}\tau^{(r-2)(k-1)}\Gamma_l;
		\end{split}		
	\end{equation*}
	therefore it remains to prove that the coefficient
	$$
	\binom{rk}{l}\frac{(N-2k+1)!}{(N-rk)!r!}
	$$
	is an integer for every $r \geq 2$ and any $l\leq k-1$. First of all, we notice that this is the same as
	$$\binom{rk}{l}\binom{N-rk+r}{r}\frac{(N-2k+1)!}{(N-rk+r)!}$$
	which implies that for $r\geq 3$ and $l\leq k-1$ this is an integer. It remains to prove the statement for $r=2$, i.e. to prove that the number 
	$$ \binom{2k}{l}\frac{N-2k+1}{2}$$
	is an integer. Notice that it is clearly true for $l\geq 1$ but not for $l=0$. However, we have that $\Gamma_0 \in 2\cdot I$ by \Cref{rem:nec}, therefore we are done.
\end{proof}

\begin{corollary}
	The ideal generated by the relations induced by $\Delta^k$ in $\PP^N$ is generated by the two elements $\pi_{1,*}(1)$ and $\pi_{1,*}(h_0)$. See \Cref{rem:gener} for the explicit description.
\end{corollary}

\bibliographystyle{plain}
\bibliography{Bibliografia}

\end{document}